\documentclass[a4paper]{amsart}
\usepackage[utf8]{inputenc}
\usepackage{amssymb, mathrsfs}
\usepackage{graphicx,subfig,color,float}
\usepackage[T1]{fontenc}
\def\factordesubfig{0.45}
\newcommand{\subfigdef}[3][\factordesubfig]{\xdef\factordesubfig{#1}%
   \subfloat[#3]{\includegraphics[width=#1\textwidth]{#2}}%
}
\makeatletter
\newcommand{\TeoremaAmbFinalMarcat}[1]{%
   \expandafter\gdef\csname end#1\endcsname{%
      \leavevmode\unskip\penalty9999
\hbox{}\nobreak\hfill\quad\hbox{\vrule width.4em
height.46em depth0pt}%
      \@endtheorem%
   }
}
\makeatother
\newtheorem{theorem}{Theorem}[section]
\newtheorem{proposition}[theorem]{Proposition}
\newtheorem{lemma}[theorem]{Lemma}

\theoremstyle{definition}
\newtheorem{remark}[theorem]{Remark} \TeoremaAmbFinalMarcat{remark}
\newtheorem{example}[theorem]{Example} \TeoremaAmbFinalMarcat{example}
\newtheorem{algorithm}[theorem]{Algorithm}
\TeoremaAmbFinalMarcat{algorithm}
\newtheorem{definition}[theorem]{Definition}
\TeoremaAmbFinalMarcat{definition}

\usepackage{enumerate}
\makeatletter
\renewcommand\theenumi{\@alph\c@enumi}
\renewcommand\theenumii{\@alph\c@enumii}
\renewcommand\theenumiii{\@alph\c@enumiii}
\renewcommand\theenumiv{\@alph\c@enumiv}

\def\@map#1#2[#3]{\mbox{$#1 \colon #2 \longrightarrow #3$}}
\def\map#1#2{\@ifnextchar [{\@map{#1}{#2}}{\@map{#1}{#2}[#2]}}
\makeatother

\DeclareMathOperator{\supp}{supp}

\DeclareMathOperator*{\esssup}{ess\,sup}
\newcommand{\evalat}[1]{\bigr\rvert_{#1}}
\newcommand{\set}[2]{\ensuremath{\left\{#1 \,\colon #2\right\}}}
\newcommand{\Rot}[2]{R_{#1}^{#2}}
\newcommand{\SP}[1]{\left\langle #1 \right\rangle}
\newcommand{\abs}[1]{\left| #1 \right|}
\newcommand{\norm}[1]{\left\| #1 \right\|}
\newcommand{\step}[2][Step]{\par\par\medskip\par\par
  \noindent\textbf{#1\ #2.}}

\newcommand{\Fr}{\mathfrak{F}}
\newcommand{\N}{\mathbb{N}}
\newcommand{\Z}{\mathbb{Z}}
\newcommand{\R}{\mathbb{R}}

\newcommand{\SI}{\mathbb{S}^1}
\newcommand{\Lii}{\mathscr{L}^{2}}
\newcommand{\Liin}{\mathscr{L}^{\infty}}
\newcommand{\Li}{\mathscr{L}^{1}}
\newcommand{\V}{\mathcal{V}}
\newcommand{\W}{\mathcal{W}}
\renewcommand{\S}{\mathcal{S}}

\newcommand{\Ccla}[1]{\mathcal{C}^{#1}}
\newcommand{\Bes}[1][s]{\mathscr{B}^{#1}_{\infty,\infty}}
\newcommand{\Reg}{\mathrm{R}}%
\newcommand{\zeron}[1]{\mathbf{Z}_{({#1})}}
\newcommand{\PER}[1]{{#1}^{\scriptscriptstyle\mathrm{PER}}}
\newcommand{\Fse}{\mathfrak{F}_{\sigma,\varepsilon}}%
\title[Estimate of the regularity for \textsf{SNA}s]{A numerical
estimate of the regularity of a family of Strange Non--Chaotic
Attractors}
\author[Ll. Alsed\`a, J.M. Mondelo and D. Romero]{
Llu\'{\i}s Alsed\`a i Soler\\
Josep Maria Mondelo Gonz\'alez\\ and \\ David Romero i 
S\`anchez$^\ast$.}
\address{Departament de Matem\`atiques, Edifici Cc, Universitat
Aut\`onoma de Barcelona, 08913 Cerdanyola del Vall\`es, Barcelona,
Spain}
\email{alseda@mat.uab.cat, jmm@mat.uab.cat, dromero@mat.uab.cat}
\thanks{The authors have been partially supported by MINECO grant
numbers MTM2008-01486, MTM2011-26995-C02-01 and MTM2014-52209-C2-1-P}
\date{\today}
\begin{document}
\begin{abstract}
We estimate numerically the regularities of a family of Strange
Non--Chaotic Attractors related with one of the models studied in
\cite{GOPY} (see also \cite{Kell}). To estimate these regularities we
use wavelet analysis in the spirit of \cite{LlaPe} together with some
ad-hoc techniques that we develop to overcome the theoretical
difficulties that arise in the application of the method to the
particular family that we consider. These difficulties are mainly due to
the facts that we do not have an explicit formula for the attractor and
it is discontinuous almost everywhere for some values of the parameters.
Concretely we propose an algorithm based on the Fast Wavelet Transform.
Also a quality check of the wavelet coefficients and regularity
estimates is done.
\end{abstract}
\maketitle
\begingroup
\def\thempfn{$\ast$\ }
\footnotetext{%
Corresponding author.
Address: Departament de Matem\`atiques,
Edifici Cc,
Universitat Aut\`onoma de Barcelona,
08913 Cerdanyola del Vall\`es,
Barcelona,
Spain.
\emph{Tel.:} (+34)935813071
\emph{Fax:} (+34)935812790
\emph{E-mail address:} {\tt dromero@mat.uab.cat}
}
\endgroup
\section{Introduction}

The aim of this paper is to develop techniques and algorithms to
compute approximations of (geometrically) extremely complicate
dynamical invariant objects by means of wavelet expansions.
Moreover, from the wavelet coefficients we want to derive an
estimate of the regularity of these invariant objects.
In the case when the theoretical regularity is known,
the comparison between both values gives a natural and good quality
test of the algorithms and approximations.

In this paper the invariant objects that we study and consider when
developing our algorithms are Strange Non-chaotic Attractors.
They appear in a natural way in families of quasiperiodically forced
skew products on the cylinder of the form
\begin{equation}\label{skpgen}
\begin{array}{r c c c}
\Fr_{\sigma,\varepsilon}:&\SI\times\R&
\xrightarrow{\phantom{xxxxxxxxx}}&\SI\times\R\\
&(\theta,x) &\longmapsto&
(\Rot{\omega}{}(\theta),F_{\sigma,\varepsilon}(\theta,x)),
\end{array}
\end{equation}
where
{\map{F_{\sigma,\varepsilon}}{\SI\times\R}[\R]} is continuous and
$\mathcal{C}^1$ with respect to the second variable,
$\Rot{\omega}{}(\theta)=\theta + \omega \pmod{1}$ with $\omega \in \R 
\setminus \mathbb{Q},$
$\SI=\R/\Z=[0,1)$ denotes the circle and $\varepsilon,\sigma\in\R^+.$
These systems have the important
property that any fibre, $\{\theta\}\times\R$, is mapped into another
fibre, $\{\Rot{\omega}{}(\theta)\}\times\R$.

Our main goal will be to derive approximations in terms of wavelets
of the invariant maps {\map{\varphi}{\SI}[\R]}:
$\varphi(\Rot{\omega}{}(\theta))= 
F_{\sigma,\varepsilon}(\theta,\varphi(\theta))$.
Under certain conditions the graphs of these invariant maps have very 
complicate geometry
where roughly speaking, the word \emph{complicate} means non-piecewise 
continuous.
In such case, we will say that the graph of
$\varphi$ is a \emph{Strange Non-chaotic Attractor (\textsf{SNA})}.
A usual particular case of \textsf{SNA} is when the
invariant function is positive in a set of full Lebesgue measure and 
vanishes
on a residual set.

A usual standard approach is to use Fourier expansions (rather than 
Wavelet ones)
when approximating dynamical invariant objects.
In the \textsf{SNA}'s framework this approach has a serious drawback:
an accurate approximation of $\varphi$ demands a high number of Fourier 
modes due the appearance
of strong oscillations (see e.g.~\cite{jorba}). One natural way to 
overcome this problem is by
using other orthonormal basis such as wavelets and the multi-scale
methods (see e.g.~\cite{Cohen,Mallat}). One of the advantages of
this approach is that wavelets also define certain regularity spaces
$\Bes[s]$ (see e.g~\cite{HeWe,Cohen,Meyer,Tri02}) that provide a natural
framework for the approximations that one gets.

Precisely, the regularity can be considered as a trait of how $\varphi$
becomes strange in terms of functional spaces. For example, in
\cite{LlaPe}, the authors make numerical implementations of wavelet
analysis to estimate the ``\emph{positive}'' regularity of invariant
objects which are graphs of functions in appropriate spaces. However,
due to the complexity of the \textsf{SNA}s described above we need to
consider the possibility that these objects have zero or even negative
regularity (see \cite{Cohen}). Hence, the techniques of \cite{LlaPe}
need to be extended to this case. To this end, we develop ad-hoc
techniques to overcome the theoretical difficulties of the objects we
study in performing a wavelet analysis, in the same spirit of
\cite{LlaPe}, to estimate the regularity of such $\varphi$. Our
wavelet analysis will be based on the Fast Wavelet Transform (see
e.g~\cite{Mallat}).

The computation of the regularity (depending on parameters) can give
some insight in the study of the fractalization or other routes of
creation of \textsf{SNA} and help in detecting this bifurcation.

We apply the above program to a slight modification of the system
considered in~\cite{GOPY}. Indeed, the
attractor obtained in \cite{GOPY} (as shown by Keller in \cite{Kell}),
is the graph of an upper semi-continuous function from the circle to
$\R$ in the \emph{pinched case} (that is, when there exists a fibre
whose image is degenerate to a point), whereas in the non pinched one
the attractor is the graph of a map with the same regularity as the skew
product (see also \cite{Star2} and \cite{Star1}).
As we will see, the wavelet coefficients together with the computed 
regularity
detect well the functional space jump associated to the creation of the 
\textsf{SNA}.

This paper is organized in two parts. The first one is devoted to make a
survey on wavelets and regularity. Whereas the second one is devoted
to apply these techniques to the \textsf{SNA} case.
More concretely, in Section~\ref{sec::ondetes} we recall some topics 
about
the theory of wavelet bases. Section~\ref{sec::onthenotion} is devoted
to review the notion of regularity through Besov functional spaces and
discuss it by means of a particular simple example. In Besov spaces the
regularity can be any real number (in contrast to H\"older regularity
defined only for positive regularities). In
Section~\ref{sec::WaveletsRegularity} we review the relation between the
regularity and the wavelet coefficients of a function.
Section~\ref{sec::TheAlgorithm} is devoted to present and test a
methodology to numerically estimate regularities based on the previous 
sections.

Finally, in the second part, in Section~\ref{sec::Statement} we survey
on the family of Strange Non--Chaotic Attractors that we will study.
In particular, we state Keller's Theorem and we remark crucial aspects
of its proof that will be used in devising the algorithm that we 
propose.
In Section~\ref{sec::extended_alg}, we present how one can
overcome some of the theoretical difficulties related with the 
strangeness
of the \textsf{SNA}. In~\ref{sec::arasi_alg}, we perform the algorithm 
to compute the
regularity of the attractors and in Section~\ref{sec::Conclusions}, the 
results of this
computation, for a particular instance of \textsf{SNA}'s, are presented
and discussed.
\section{A survey on wavelets}\label{sec::ondetes}
We aim at approximating by means of wavelets a certain class of
functions from the circle $\R\setminus\Z$ to an interval of the real
line. Recall that a standard approach used in the literature to compute
and work with invariant objects of systems exhibiting periodic or
quasi-periodic behaviour is to use finite Fourier approximations
(trigonometric polynomials), namely functions of the form
\[
\varphi(\theta)=a_0+\sum_{n=1}^{N}
\left(a_n\cos(n\theta)+b_n\sin(n\theta)\right).
\]
As it has been said, in this paper instead we aim
at using finite wavelet expansions of the form:
\[
\varphi(\theta)=a_0+\sum_{j=0}^{N}\sum_{n=0}^{N_j}d_{j,n}\psi_{j,n}
(\theta),
\]
where $\psi_{j,n}(\theta)$ is obtained by translation and dilation of a
mother wavelet $\psi(x)$. To be explicit, let us start by introducing
the orthonormal wavelet basis of $\Lii(\R)$. A natural way to do it is
via the notion of Multiresolution Analysis. We refer the reader to
\cite{Mallat, HeWe} for more detailed and comprehensive expositions.
\begin{definition}\label{def08}
A sequence of closed subspaces $\{\V_{j}\}_{j\in\Z}$ of $\Lii(\R)$ is a
\emph{Multiresolution Analysis} (or simply a \emph{MRA}) if it satisfies
the following six properties:
\begin{enumerate}
\item
$
 \{0\} \subset \dots \subset \V_{1} \subset \V_{0}
       \subset \V_{-1} \subset \dots \subset \Lii(\R).
$
\item  $\{0\}=\bigcap_{j\in\Z}\V_{j}.$
\item $\mathrm{clos}\left( \bigcup_{j\in\Z}\V_{j} \right) = \Lii(\R).$
\item There exists a function $\phi$ whose integer translates,
$\phi(x-n)$, form an orthonormal bases of $\V_{0}$. Such function is
called the \emph{scaling function}.
\item For each $j\in\Z$ it follows that $f(x)\in\V_{j}$ if and only if
$f(x-2^{j}n)\in\V_{j}$ for each $n\in\Z$.
\item For each $j\in\Z$  it follows that $f(x)\in\V_{j}$ if and only
if $f(x/2)\in\V_{j+1}$.
\end{enumerate}
\end{definition}
Before continuing the explanation, let us recall that for
$f\in\Lii(\R)$,
\[
\widehat{f}(\xi) = \int_{\R} f(x)e^{-i\xi x}\ dx,
\xi\in\R,
\]
denotes the \emph{Fourier transform} of $f$ and $f^{\vee}(x)$
\[
f^{\vee}(x) = \dfrac{1}{2\pi} \int_{\R} f(\xi)e^{i\xi x}\ d\xi,
x\in\R
\]
stands for the \emph{inverse Fourier transform}. If we fix an MRA, it
follows that $\V_{j}$ has an orthonormal basis
$\{\phi_{j,n}\}_{n\in\Z}$, for every $j$, where
\[
\phi_{j,n}(x) = 2^{-j/2} \phi\left( \frac{x-2^{j}n}{2^{j}} \right).
\]
Now, define the subspace  $\W_{j}$ as the orthogonal complement of
$\V_{j}$ on $\V_{j-1}$. That is,
\begin{equation}\label{estruct_espai}
\V_{j-1}=\W_{j}\oplus\V_{j}.
\end{equation}
Therefore, by the inclusion of the spaces $\V_{j}$ we have
\begin{equation}\label{nova01}
\Lii(\R) =
 \mathrm{clos}\left(\bigoplus_{j\in\Z}\W_{j} \right)=
 \mathrm{clos}\left(\V_0 \oplus
   \bigoplus_{j=-\infty}^{0} \W_{j}\right).
\end{equation}
The \emph{mother wavelet} $\psi\in\W_{0}$ is defined to be the function
whose Fourier transform is
\begin{equation}\label{equ33}
\widehat{\psi}(\xi) = \frac{1}{\sqrt{2}} e^{-i\xi}
      \widehat{h}^{*}(\xi+\pi)\widehat{\phi}(\xi)
\end{equation}
where $\widehat{h}^{*}(\xi)$ is the complex conjugate of
\begin{equation}\label{tonta}
\widehat{h}(\xi)=\sum_{n\in\Z}h[n]e^{-in\xi},
\end{equation}
with $\widehat{h}(0)=\sqrt{2}$ and
$
h[n] = \SP{\frac{1}{\sqrt{2}} \phi\left(\frac{x}{2}\right),
\phi(x-n)}$ for $n\in\Z$. The sequence $h[n]$ is called the
\emph{scaling filter} (or the \emph{low pass filter}) of the
Multiresolution Analysis. We define the support of $h[n]$, denoted by
$\supp(h)$, as the minimum subset $\mathfrak{I}$ of $\mathbb{Z}$ such
that $\mathfrak{I}=\{\ell,\ell+1,\dots, \ell'\}$ is a set of consecutive
integers and
$$h[n]=0\ {\text{for every $n\in\mathbb{Z}\backslash\mathfrak{I}$}}.$$
The following result (see \cite[Theorem~7.3]{Mallat}) allows to obtain
the wavelet basis from the scaling function:
\begin{theorem}[Mallat, Meyer]\label{Mame}
The mother wavelet given by Equation~\eqref{equ33} verifies that, for
each integer $j$, the family $\{\psi_{j,n}\}_{n\in\Z}$ is an
orthonormal basis of $\W_{j}$, where:
\[
\psi_{j,n}(x) = 2^{-j/2}\psi\left(\frac{x-2^{j}n}{2^{j}}\right).
\]
As a consequence, the family $\{\psi_{j,n}\}_{(j,n)\in\Z\times\Z}$ is
an orthonormal basis of $\Lii(\R)$.
\end{theorem}
Recall that we want to approximate maps $f \in \Lii(\R)$ by linear
combinations of wavelets. By \eqref{nova01} and the theorem above, the
projection of $f$ to $\V_{-J} \subset \Lii(\R):$
\[
  \sum_{n\in\Z} \SP{f,\phi_{-J,n}} \phi_{-J,n},
\]
is a good approximation of $f$  provided that $J > 0$ is large enough.
We want to rewrite such an approximation as linear combination of
wavelets of the form
\[
f \sim
  \sum_{n\in\Z} \SP{f,\phi_{0,n}} \phi_{0,n} +
  \sum_{j=0}^{J-1} \sum_{n\in\Z} \SP{f,\psi_{-j,n}} \psi_{-j,n}
\in \V_0 \oplus \bigoplus_{j=0}^{J-1} \W_j.
\]
To do it, as usual, we define the coefficients
\[
  a_{j}[n] := \SP{f,\phi_{j,n}}
  \quad\text{and}\quad
  d_{j}[n] := \SP{f,\psi_{j,n}}
\]
for $j,n\in\Z.$ With this notation, the initial approximation of $f$
becomes
\[
  \sum_{n\in\Z} a_{-J}[n] \phi_{-J,n}.
\]
By \eqref{estruct_espai} we have
\begin{equation}\label{estruct_espai_coeficients}
\sum_{n \in \Z} a_{-j}[n] \phi_{-j,n} =
\sum_{n \in \Z} a_{-j+1}[n] \phi_ {-j+1,n} +
\sum_{n \in \Z} d_{-j+1}[n] \psi_{-j+1,n}
\end{equation}
for every $j \in \Z.$

To obtain the coefficients $a_{-j+1}[n]$ and $d_{-j+1}[n]$ from
$a_{-j}[n],$ we use the \emph{Fast Wavelet Transform (FWT)} given by
(see \cite[Theorem~7.7]{Mallat}):
\begin{equation}\label{FWT}
\left\{\begin{aligned}
& a_{j+1}[p] := \sum_{n\in\Z} h[n-2p] a_{j}[n]
      \quad\text{and}\quad
      d_{j+1}[p] := \sum_{n\in\Z} g[n-2p] a_{j}[n]; \\
& \text{with}\ g[p] = (-1)^{1-p}h[1-p]
\end{aligned}\right.
\end{equation}
for every $j,p\in\Z.$ Hence, from the iterative use of
\eqref{estruct_espai_coeficients} and \eqref{FWT} starting with the
approximation $\sum_{n\in\Z} a_{-J}[n] \phi_{-J,n}$ we obtain the
approximation of $f$ that we are looking for:
\[
f \sim
  \sum_{n\in\Z} a_0[n] \phi_{0,n} +
  \sum_{j=0}^{J-1} \sum_{n\in\Z} d_{-j}[n] \psi_{-j,n}
\in \V_0 \oplus \bigoplus_{j=0}^{J-1} \W_j.
\]

For (numerical) applications such infinite approximations are usually
not available since we often work with finite information about our
function. For this we need a similar theory for subspaces of $\V_j$ and
$\W_j$ of finite dimension. For $j \ge 0$ we define
\begin{align*}
\V_{-j}^*:= & \SP{\phi_{-j,0}, \phi_{-j,1}, \dots, \phi_{-j,2^j-1}}
	      \subset \V_{-j},\ \text{and}\\
\W_{-j}^* := & \SP{\psi_{-j,0}, \psi_{-j,1}, \dots, \psi_{-j,2^j-1}}
	      \subset \W_{-j},
\end{align*}
where $\SP{f_1,f_2,\dots,f_n}$ denotes the subspace of $\Lii(\R)$
generated by the linear combinations of $f_1,f_2,\dots,f_n$. From the
comment at the end of Section~7.3.1 of \cite{Mallat} (see also
\cite[Lemma~3.26]{Frazier} for a more detailed account), it follows that
\begin{equation}\label{estruct_espai_finit}
\V_{-j+1}^*=\W_{-j}^*\oplus\V_{-j}^*
\end{equation}
for every $j > 0$. Hence, given a function $f \in \Lii(\R)$ we can take
a good finite approximation of the map given by its projection to
$\V_{-J}^*:$
\begin{equation}\label{initial-projection}
f \sim \sum_{n=0}^{2^{J}-1} a_{-J}[n] \phi_{-J,n},
\end{equation}
provided that $J$ is large enough. Again, we are interested in writing
such an approximation as linear combination of wavelets, but in this
case this expansion must be finite:
\[
f \sim
  a_0[0] \phi_{0,0} +
  \sum_{j=0}^{J-1} \sum_{n=0}^{2^j-1} d_{-j}[n] \psi_{-j,n}
\in \V_0^* \oplus \bigoplus_{j=0}^{J-1} \W_{-j}^*.
\]

To obtain this expression observe that \eqref{estruct_espai_finit}
implies
\begin{equation}\label{estruct_espai_finit_coeficients}
\sum_{n=0}^{2^{j}-1} a_{-j}[n] \phi_{-j,n} =
\sum_{n=0}^{2^{j-1}-1} a_{-j+1}[n] \phi_ {-j+1,n} +
\sum_{n=0}^{2^{j-1}-1} d_{-j+1}[n] \psi_{-j+1,n}
\end{equation}
for $j > 0.$ Now, to obtain the coefficients $a_{-j+1}[n]$ and
$d_{-j+1}[n]$ from $a_{-j}[n],$ instead of using formulae \eqref{FWT},
we use the following circular convolution version of them (see
\cite[Section~7.5.1]{Mallat} or the proof of
\cite[Lemma~3.26]{Frazier}):
\begin{equation}\label{FWT_finit}
\left\{\begin{aligned}
& a_{-j+1}[p] := \sum_{n=0}^{2^{j}-1} h[n-2p] a_{-j}[n]
      \quad\text{and}\quad
      d_{-j+1}[p] := \sum_{n=0}^{2^{j}-1} g[n-2p] a_{-j}[n]; \\
& \text{with}\ g[p] = (-1)^{1-p}h[p]
\end{aligned}\right.
\end{equation}
for every $j > 0$ and $p\in\{0,1,\dots,2^{j-1}-1\}.$ Hence, with the
iterative use of \eqref{estruct_espai_finit_coeficients} and
\eqref{FWT_finit} starting with the approximation
\eqref{initial-projection} we obtain
\begin{equation}\label{final_approximation}
f \sim
  a_0 + \sum_{j=0}^{J-1} \sum_{n=0}^{2^j-1} d_{-j}[n] \psi_{-j,n}
\in \V_0^* \oplus \bigoplus_{j=0}^{J-1} \W_{-j}^*.
\end{equation}
as we wanted.

\begin{remark}
It is important to point out that the ``\emph{finiteness}'' of the FWT
turns the orthonormal basis of $\W_{-j}^*$ into an orthonormal basis of
$\SI$ for $j>0$. Therefore, we do not need anything more when we have
to deal with maps which naturally live in $\SI$.
\end{remark}

To effectively compute an approximation of the type given in
Equation~\eqref{final_approximation} one remaining problem is left:
to find a good estimate of the initial coefficients $a_{-J}[n] =
\SP{f,\phi_{-J,n}}$. In the literature there is a lot of discussion on
how to compute these coefficients, but a simple customary approach is to
use the following estimate (see, for instance,
\cite[Lemma~5.54]{Frazier} and its proof):
\begin{lemma}\label{FWT-InApprox}
Assume that $f$ verifies $|\langle f,\phi_{j,n}\rangle|<\infty$ for
every $j,n\in\Z\times\Z$ and
\[
  \abs{f(x)-f(y)} \leq C_1 \abs{x-y}^{\alpha} \text{ with
$\alpha\in(0,1]$}
\]
for all real numbers $x,y$ and a constant $C_1 < \infty$. Suppose that
the scaling function $\phi$ from an MRA $\{\V_{j}\}_{j\in\Z}$ is such
that
\[
\phi \in \mathscr{L}^{1}(\R),\
\widehat{\phi}(0) = \int_{\R}\phi(x)\ dx=1
\text{ and }
\int_{\R} \abs{x}^{\alpha} \phi(x)\ dx < C_2.
\]
Then, for every $j,n\in\Z\times\Z$,
\[
\abs{\SP{f,\phi_{j,n}} - 2^{j/2}f(2^{j} n)} <
  C_1C_2 2^{j\left(\alpha+\tfrac{1}{2}\right)}.
\]
\end{lemma}

As a corollary of this lemma we see that if $f$ is Lipschitz, then
\[
  a_{-J}[n] \approx 2^{-J/2}f(2^{-J}n).
\]

Summarizing, Lemma~\ref{FWT-InApprox} gives us a method to initialize
the FWT. This gives a complete algorithm to compute wavelet
coefficients of a certain function $f\in\Lii(\R)$.
\section{Defining regularity through Besov
spaces}\label{sec::onthenotion}
In this section we will make precise the notion of regularity that we
will use. To do so, we will describe, in two steps, the functional
spaces that define the notion of regularity. Roughly speaking these
spaces collect the functions that verify an $\alpha$-H\"{o}lder
condition. However, As we will see, we will have to deal with functions
with regularity zero (that do not verify any H\"{o}lder condition).
The notion of non-positive regularity is formalized trough Besov spaces
(see \cite{Tri01, BeLo}).
In what follows, recall the definition of the Besov spaces on the real
line \cite[Section~2.3]{Tri01}. Next we will recall the extension of
such definition to $\SI$.

\subsection{The spaces $\Bes(\R)$}\label{sec::Bes}
The space of all real valued rapidly decreasing infinitely
differentiable functions is called the \emph{(real) Schwartz space} and
it is denoted by $\S(\R)$. The topological dual of $\S(\R)$ is the space
of \emph{tempered distributions} which is denoted by $\S'(\R)$. For
$f\in\S'(\R)$, $\widehat{f}(\xi)$ denotes the \emph{Fourier transform}
of $f$ and $f^{\vee}(x)$ stands for the \emph{inverse Fourier transform}
in the sense of distributions (see e.g~\cite{Tri01}). Recall, also, that
the essential supremum is defined as
\[
\esssup_{x\in\R}f(x)=\inf \{a \in \mathbb{R}: \mu(\{x\in\R: f(x) > a\})
=0\},
\]
where $\mu$ is a measure (in our case the usual Lebesgue measure).

Let $\varphi_{0}\in\S(\R)$ be such that
\[
\varphi_{0}(x):=\begin{cases}
1 & \text{if $\abs{x}\leq 1$} \\
0 & \text{if $\abs{x}\geq 3/2$}
\end{cases}
\]
and set
\[
\varphi_{j}(x) := \varphi_{0}(2^{-j}x) - \varphi_{0}(2^{-j+1}x)
\]
for $j\in\N$. It is not difficult to show that, independently of the
choice of $\varphi_0$, $\sum_{j=0}^{\infty}\varphi_{j}(x) = 1$ for all
$x\in\R.$  Each of the families $\{\varphi_{j}\}_{j=0}^{\infty}$ is
called a \emph{Dyadic Resolution of Unity} in $\R$.
\begin{definition}\label{espaibes}
Let $\varphi=\{\varphi_{j}\}_{j=0}^{\infty}$ be a Dyadic Resolution of
Unity and $s\in\R$. For $f\in\S'(\R)$ we define the quasi-norm
\[
\norm{f}_{\infty,\infty,\varphi,s} =
\sup_{j\geq0} 2^{js}
   \left( \esssup_{x\in\R} \abs{(\varphi_{j}\widehat{f})^{\vee}(x)}
\right).
\]
Then, we define the \emph{Besov Spaces} by
\[
\Bes(\R) :=
  \set{f\in\S'(\R)}{\norm{f}_{\infty,\infty,\varphi,s} <\infty }.
\]
\end{definition}

As it can be seen in \cite[Remark~2 of Section~2.3]{Tri01}, the spaces
$\Bes(\R)$ are, in fact, independent of the chosen dyadic resolution of
unity $\varphi$. Therefore, we can remove the subscript $\varphi$ from
$\norm{f}_{\infty,\infty,\varphi,s}$. So, in what follows we will write
$\norm{f}_{\infty,\infty,s}$ instead of
$\norm{f}_{\infty,\infty,\varphi,s}$. The spaces $\Bes(\R)$ are a
particular case of the Generalized Besov Spaces
$\mathscr{B}_{p,q}^{s}(\R)$ defined also, for example, in \cite{Tri01}
and one has the inclusion property. That is, if $s<s',$ then
$\mathscr{B}_{p,q}^{s'}(\R) \subset \mathscr{B}_{p,q}^{s}(\R)$.

For $s > 0$, the spaces $\Bes(\R)$ coincide with the H\"older-Zygmund
spaces and it is natural to extend the notion of regularity to $s\leq 0$
through $\Bes(\R)$ in the following way (we refer to \cite{Tri01, Stein}
for a more complete explanation).

\begin{definition}\label{alphaneg}
We say that a map $f$ \emph{has regularity} $s\in\R$ if $f\in\Bes(\R)$.
\end{definition}

\begin{example}\label{alphanegex}
The following examples help to clarify this regularity notion.
\begin{enumerate}[(i)]
\item Consider the \emph{Weierstra\ss\ function} defined by
\[
\mathfrak{W}_{A,B}(x) := \sum_{n=1}^{\infty} A^{n}\sin(B^nx),
\]
where $A,B\in\R$ are such that $B^{-1} < A < 1 < B$ is H\"older
continuous and nowhere differentiable (but it has a distributional
derivative, as does every locally integrable function). Moreover, it has
regularity $-\log_{B}(A)$; that is $\mathfrak{W}_{A,B} \in
\Bes[-\log_{B}(A)](\R)$.

\item The function $f(x)=\frac{-1}{\log(|x|)}$ (with $f(0)=0$) belongs
to $\Bes[0]$ in a neighbourhood of $x=0$ since $F(x)$, where
$F'(x)=f(x)$, is Lipschitz (because is the anti-derivative of a bounded
function) and the derivative operator reduces $s$ by 1 (leaving
$p=q=\infty$ unchanged). As matter of fact, \cite[Section 2.3, Example
1]{RuSi} is devoted to give examples of such functions (in terms of
belonging to a certain functional spaces). The prototypical example is
to consider $\alpha^2+\beta^2>0$, with $\beta>0$, and defining \[
f_{\alpha,\beta}(x)=\upsilon(x)|x|^{\alpha}(-\log{|x|})^{-\beta} \]
where $\upsilon(x)$ is a smooth cut-off function with
$\supp\upsilon\subset\{x\in\R\ \colon\ |x|\leq \delta\}$ and $\delta>0$.
That is, $\upsilon(x)$ has the support near the origin and the
singularity is located near the origin.\label{ex2}

\item The negative exponents, in the Besov spaces, must be understood as
a distributional space. For example, it is known that
$\delta(x)\in\Bes[-1](\R)$ where $\delta(x)$ stands for Dirac's delta.
In view of that, $\delta(x)$ can be considered as the second
(distributional) derivative of the continuous function
\[
f(x)=\begin{cases}
      0 & \text{if $x < 0,$}\\
      x & \text{if $x \geq 0.$}\\
     \end{cases}
\]\label{ex3}
\end{enumerate}
\end{example}

We conclude this section with some remarks on the spaces $\Bes[s]$, with
$s\in[0,1)$. The fact that $f$ belongs to a concrete H\"older space is
equivalent to specify its differentiability degree and how
``\emph{wild}'' is the last derivative. Indeed, if a function $f$ is
Hölder $n+s$, where $n\in\mathbb{N}$ and $s\in(0,1)$, then it can be
approximated by a polynomial of order $n$ (the Taylor polynomial for
example) and the difference between the polynomial and the function is
uniformly bounded by an $|y|^{s}$ factor. Moreover, if $n\geq1$ then $f$
is continuously differentiable whereas, if $n=0$ then $f$ ``\emph{may be
taken}'' to be continuous (but not differentiable) in the spirit of the
following proposition.

\begin{proposition}[Proposition V.4.6 \cite{Stein}]
Every $f\in\Bes[s](\R)$, with $s\in(0,1)$ may be modified on a set of
measure zero so that it becomes continuous.
\end{proposition}

Finally, we want to mention that in \cite[Section~3.8]{Cohen} it is
shown that the range of negative $s$ is the dual of the positive ones.
That is, for $s<0$ the spaces $\Bes[s](\SI)$ are ``\emph{purely}''
distributional spaces (see Example~\ref{alphanegex}~\eqref{ex3}). Thus,
we can state the following lemma.

\begin{lemma}\label{salt}
Let $f\in\Lii(\R)$ be an upper semi-continuous function which is not
continuous. Then $f\in\Bes[0](\R)$ (that is, $f$ has zero regularity).
\end{lemma}

\begin{proof}
If $f$ is an upper semi-continuous function which is not continuous
then, it cannot verify an $s$-H\"{o}lder condition (for any $s>0$). 
Thus,
the regularity $s$ is non-positive. On the other side, since $f$ is
not a distribution, the parameter $s$ cannot be negative.
Hence, $f\in\Bes[0](\R)$.
\end{proof}

\subsection{The Besov spaces on $\SI$}
In this section we will extend the Besov spaces to $\SI$. Recall that we
consider $\SI=\R\setminus\Z$ and, hence, as the interval $[0,1)$. To do
it we follow \cite{BeLo,Tri03}. Indeed, given $f\in\S'(\SI)$ (the space
of tempered distributions on $\SI$) it is known that
\[
f=\sum_{n\in\Z}\widehat{f}(n)e^{inx}.
\]
\begin{definition}\label{alphanegtor}
Let $\varphi = \{\varphi_{j}\}_{j=0}^{\infty}$ be a dyadic resolution
of unity (on $\R$). We define the \emph{Besov Spaces} on $\SI$ by
\[
\Bes(\SI) := \set{f\in\S'(\SI)}{\norm{f}_{\infty,\infty,s}<\infty}
\]
where
\[
\norm{f}_{\infty,\infty,s} =
  \sup_{j\geq0} 2^{js}\left( \esssup_{x\in\R} \abs{\sum_{n\in\Z}
\varphi_{j}(n)\widehat{f}(n)e^{inx}} \right)
\]
is a quasi-norm for the quasi-Banach space $\Bes(\SI)$.
\end{definition}
As in Definition~\ref{alphaneg} we say that a circle map $f$ \emph{has
regularity} $s\in\R$ if the map $f$ belongs to $\Bes(\SI)$.

The following lemma shows that the regularity of a circle map coincides
with the regularity of its real extension which we define as follows.
Given $f\in\S'(\SI)$ there exists a unique $\PER{f} \in \S'(\R)$ such
that $\PER{f}$ is 1-periodic and the restriction of $\PER{f}$ over
$[0,1)$ coincides with $f$ (such an $\PER{f}$ can be defined as
$f(\{\cdot\})$, where $\{\cdot\}$ denotes the fractional part function).
This lemma is usually omitted and used implicitly but we include here
for completeness.
\begin{lemma}\label{PER}
For every $f\in\S'(\SI)$ it follows that $\PER{f}\in\Bes(\R)$ if and
only if $f\in\Bes(\SI)$.
\end{lemma}

\begin{proof}
Since $\PER{f}$ is 1-periodic and $\PER{f}\evalat{[0,1]} = f$,
\[
\widehat{\PER{f}}(n) =
 \int_{0}^{1} \PER{f}(x) e^{-2\pi inx}\ dx =
 \int_{0}^{1} f(x) e^{-2\pi inx}\ dx =
 \widehat{f}(n).
\]
Hence,
\begin{align*}
\esssup_{x\in\R} \abs{\sum_{n\in\Z} \varphi_{j}(n) \widehat{f}(n)
e^{inx}}
&= \esssup_{x\in\R} \abs{\sum_{n\in\Z} \varphi_{j}(n)
\widehat{\PER{f}}(n) e^{ inx}} \\
&= \esssup_{x\in\R} \abs{\sum_{n\in\Z} (\varphi_{j}\widehat{\PER{f}})(n)
e^{inx}}\\
&= \esssup_{x\in\R} \abs{(\varphi_{j}\widehat{\PER{f}})^{\vee}(x)}.
\end{align*}
That is,
$\norm{\PER{f}}_{\infty,\infty,s} = \norm{f}_{\infty,\infty,s}$
and, hence,
$\PER{f}\in\Bes(\R)$ if and only if $f\in\Bes(\SI)$.
\end{proof}
\section{Wavelets and regularity}\label{sec::WaveletsRegularity}

In Section~\ref{sec::onthenotion} we have recalled the notion of the
regularity of a function through the spaces $\Bes(\R)$ and $\Bes(\SI)$.
Also, we have introduced the wavelet expansions of a given function in
$\Lii(\R)$. Next, we want to show the relationship between this notion
of regularity and the wavelet coefficients. Such relationship will be
the main tool of the forthcoming Algorithm~\ref{algfinal}. The main tool
for this will be the Daubechies wavelets, because they are orthonormal
bases on $\Lii(\R)$ (see \cite{Mallat} for a definition and
construction) and, depending on the number of vanishing moments, they
are well adapted to the functional spaces $\Bes(\R)$ (see
\cite{HeWe,Tri02}).

\begin{definition}\label{def15}
Let $\psi(x)$ be a wavelet from a MRA $\{\V_{j}\}_{j\in\Z}$. We say that
$\psi$ has $p$-\emph{vanishing moments} if the integer $p$ is the
minimum non-negative integer such that
\[
 \int_{\R} x^{k} \psi(x)\ dx=0 \text{ for $0\leq k<p$.}
\]
\end{definition}

Daubechies wavelets are a family of wavelets with compact support that
have an element with $p$ vanishing moments for each $p \ge 1$.
From \cite[Theorem~7.16]{HeWe},~\cite[Theorem~3.8.1]{Cohen} and
\cite[Theorem~1.64]{Tri02} we will state the following theorem, in the
spirit of~\cite[Theorem~5.10]{LlaPe} and \cite[Chapter~3]{Meyer}, which
will be useful for our purposes.
To this end, for $t\in \R$ we set
\[
\Reg(t) =
\begin{cases}
t - \tfrac{1}{2} & \text{if $t > \tfrac{1}{2},$}\\
t + \tfrac{1}{2} & \text{if $t < -\tfrac{1}{2}$}\\
0                & \text{if $t \in \left[-\tfrac{1}{2}, 
\tfrac{1}{2}\right]$.}
\end{cases}
\]

\begin{theorem}\label{util}
Let $f \in \Lii(\R)$ and let $\psi$ be a mother Daubechies wavelet with
more than $\max(\Reg(\tau),\tfrac{5}{2} - \Reg(\tau))$ vanishing moments
for some $\tau \in \R\setminus\left[-\tfrac{1}{2}, \tfrac{1}{2}\right].$
Then, $f\in\Bes[\Reg(\tau)](\R)$ if and only if there exists $C > 0$ 
such that
\[
\sup_{n\in\Z} \abs{\SP{f,\psi_{j,n}}} \leq C 2^{\tau j}
\]
for all $j\leq 0$.
Furthermore, if $\psi$ has more than 2 vanishing moments, then
$f\in\Bes[0](\R)$ if and only if either the sequence
$
\left\{2^{-\tau j} 
\sup_{n\in\Z}\abs{\SP{f,\psi_{j,n}}}\right\}_{j=0}^{-\infty}
$
is unbounded for every $\tau \in \R$ or there exist
$C > 0$ and $\tau \in \left[-\tfrac{1}{2}, \tfrac{1}{2}\right]$ such 
that
\[
\sup_{n\in\Z} \abs{\SP{f,\psi_{j,n}}} \leq C 2^{\tau j}
\]
for $j \le 0$
\end{theorem}

\begin{remark}\label{LinearRegression}
In view of Theorem~\ref{util}, if the coefficients $\sup_{n\in\Z}
\abs{\SP{f,\psi_{j,n}}}$ decay approximately exponentially with respect
to $j$, that
is
\begin{equation}\label{LineqReg}
s_j := \log_{2} \left( \sup_{n\in\Z}
\abs{\SP{f,\psi_{j,n}}}
\right) \approx
   \tau  j + \log_{2}(C),
\end{equation}
then $f\in\Bes[\Reg(\tau)](\R)$. This tells us that, in this case, to
compute the value of regularity $s$ we can make a linear regression to
estimate the slope $\tau$ of the graph of the pairs $(j,s_j)$ and get 
the
correct value of $s$ from this slope. The Pearson correlation
coefficient controls the degree of linear correlation between the
variables $j$ and $s_j$.

Moreover, in Theorem~\ref{util} the constant $C>0$ must be understood as
the norm of $f$. Therefore, Equation~\eqref{LineqReg}, can be used to 
see
the ``\emph{jumps}'' between different Besov spaces. In other words, the
``\emph{pass}'' from a concrete Besov space to another one can be
recovered by the inspection of the values of $C$. This is because such
$C>0$ should be unbounded if, for a concrete $s$, $f\not\in\Bes[s]$.
\end{remark}

In view of the above Remark, we can perform a strategy to estimate the
regularity of a function using the wavelet coefficients. This is the
main topic of the following subsection.

\subsection{A method to estimate regularities on
$\Liin$}\label{sec::TheAlgorithm}
As we have said, we want to compute wavelet approximations of certain
dynamical objects while controlling the precision of these 
approximations
(in fact, the quality of the computed wavelet coefficients).
This quality test  will be done by comparing the theoretical regularity
of such functions with the estimated one
(given by Theorem~\ref{util} and Remark~\ref{LinearRegression}). More
concretely, in \cite{LlaPe}, numerical implementations of wavelet
analysis to estimate the ``\emph{positive}'' regularity of conjugacies
between critical circle maps are done. Due to Theorem~\ref{util}, we can
generalize such techniques to any  value (positive or not) of the
regularity measured in terms of the Besov Spaces $\Bes(\R)$. The steps
described below explain how to apply Theorem~\ref{util} in a general
framework.

Among the many methods described in the literature to compute
wavelet approximations, in this paper we will use (and test)
the Fast Wavelet Transform. Alternatively, in a forthcoming paper we 
will explore
the technique of solving numerically the Invariance Equation given in
Remark~\ref{remalg} which can be more adapted to the dynamical 
complexity
of the object.

\begin{remark}\label{pascual}
In view of Lemma~\ref{PER}, to estimate the regularity of a map $f
\in \S'(\SI)$ it is enough to use Theorem~\ref{util} for $\PER{f}$.
Moreover, if $f\in\Liin(\SI),$ then
$
 |\langle f,\psi_{j,n}\rangle|<\infty
$
for all $j,n\in\Z\times\Z$.
\end{remark}

In view of this remark, the verbatim application of Theorem~\ref{util} 
is the following:

\step{1} Use Lemma~\ref{FWT-InApprox} to compute $\PER{a}_{-J}[n]:=
\SP{\PER{f},\phi_{-J,n}}$ for $0\leq n\leq 2^{J}-1$.

\step{2} Use Equation~\eqref{FWT} to calculate the coefficients
$
\PER{d}_{-j}[n] = \SP{\PER{f},\psi_{-j,n}}
$
for $j=0,\dots,J-1$ and $0 \leq n \leq 2^{j}-1$.

\step{3} By using the coefficients $\PER{d}_{-j}[n]$ from Step~2,
calculate, in view of Theorem~\ref{util},
\[
s_{-j} = \log_{2}\left(
   \sup_{0\leq n \leq2^{j}-1} \abs{\PER{d}_{-j}[n]}
\right)
\]
for $j=0,\dots,J-1$.

\step{4} In view of Equation~\eqref{LineqReg}, make a linear regression
to estimate the slope $\tau$ of the graph of the pairs $(-j,s_{-j})$
with $j=0,\dots,J-1$. Then, when there is evidence of linear correlation
between the variables $-j$ and $s_{-j}$, we set $s = \Reg(\tau).$

\step{5} If $k > \max(s,5/2-s)$ then, by Theorem~\ref{util},
$f\in\Bes(\R)$ and, hence, $f$ has regularity $s$. Otherwise we need to
repeat Step 1 -- 4 with a Daubechies wavelet having a larger value of
$k$ until $k > \max(s,5/2-s).$

To test the precision of this implementation of Theorem~\ref{util}
we will use the Weierstra\ss{} function.
The reason for such experiment is twofold.
From one side there exists an analytic
formula for the regularity of the Weierstra\ss{} function in terms
of its parameters (which allows us to test the quality of the computed
coefficients) and, at the same time,
the graph of the Weierstra\ss{} function is ``\emph{strange}'' enough.
This idea is borrowed from \cite{LlaPe}, but since we use
more data than in~\cite{LlaPe} we reproduce the example.

\begin{example}\label{exemnumWeier}
From Section~\ref{sec::onthenotion} we know that
$\mathfrak{W}_{A,B}\in\Bes[-\log_{B}(A)](\R).$
To test the algorithm we fix the parameter $B=2$ and we take
$A \in [0.56745,0.86475]$.
Hence, $\mathfrak{W}_{A,2} \in \Bes(\R)$ with
$s = -\log_{2}(A) \in [0.2051\dots, 0.8174\dots]$.
Then, observe that
\[
1 < \max\left(s,\dfrac{5}{2} - s\right) = \dfrac{5}{2} - s < 3.
\]
Therefore the above algorithm is valid in this case only for
Daubechies wavelets with $k \ge 3$ vanishing moments.

To perform the above algorithm we take $J=30$ (that is, we use a
sample of the graph of $\mathfrak{W}_{A,2}$ of $2^{30}$ points). To
carry out Step~1, by Lemma~\ref{FWT-InApprox}, we can estimate
\[
a_{-J}[n] \approx 2^{-J/2} \mathfrak{W}_{A,2}(2^{-J}n).
\]
Then, after executing Steps~1--4 of the above algorithm we obtain the
results depicted in Figure~\ref{taulaweie}.
\begin {figure}[htbp]
\begin{center}
\begin{picture}(165,125)%
\put(5,5){\includegraphics[scale=0.2]{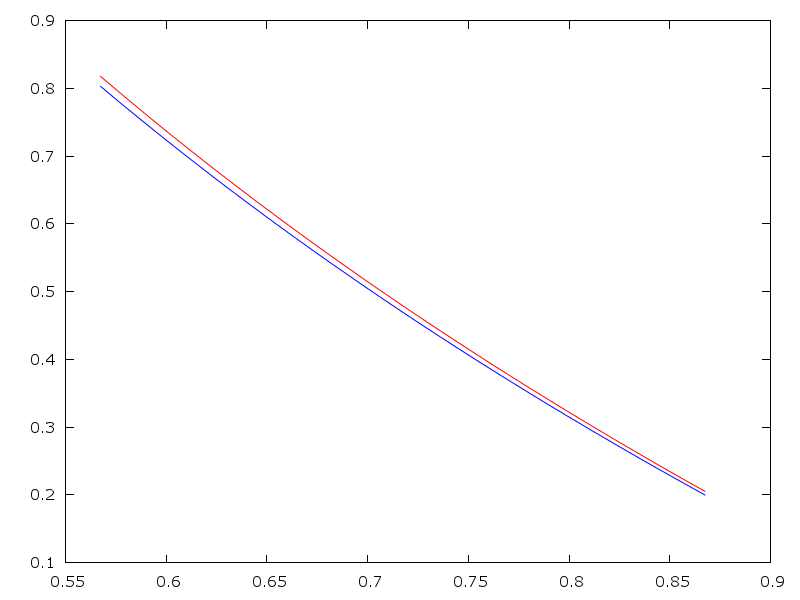}}%
\put(0,70){\rotatebox{90}{\makebox(0,0){\tiny Regularity}}}%
\put(90,0){\makebox(0,0){\tiny $A$}}%
\end{picture}
\quad
\begin{picture}(165,125)%
\put(5,5){\includegraphics[scale=0.2]{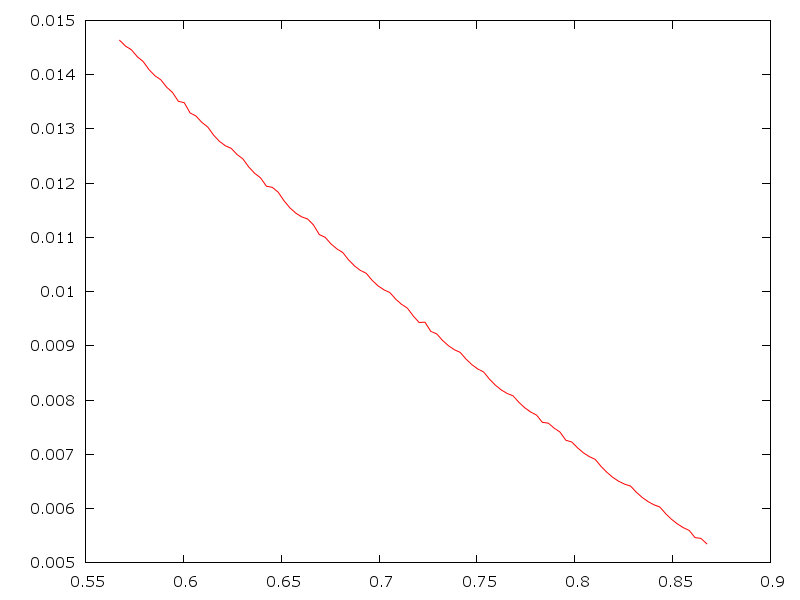}}%
\put(0,70){\rotatebox{90}{\makebox(0,0){\tiny $\bigl|-\log_{2}
(A)-s_A\bigr|$}}}%
\put(90,0){\makebox(0,0){\tiny $A$}}%
\end{picture}%
\caption{On the left picture the theoretical and estimated regularity
of $\mathfrak{W}_{A,2}$ with $A \in [0.56\cdots,0.86\cdots]$ are shown.
The theoretical curve is plotted in \textcolor{blue}{blue} and the
\emph{numerical} one in \textcolor{red}{red}. The estimated regularity
is computed with a Daubechies wavelet with 10 vanishing moments. On the
right picture the \emph{Error} function $|-\log_{2}(A)- s_A|$ is
plotted (here $s_A$ denotes the estimated regularity of
$\mathfrak{W}_{A,2}$). Notice that the error is decreasing as the
regularity gets closer to zero.}\label{taulaweie}
\end{center}
\end {figure}
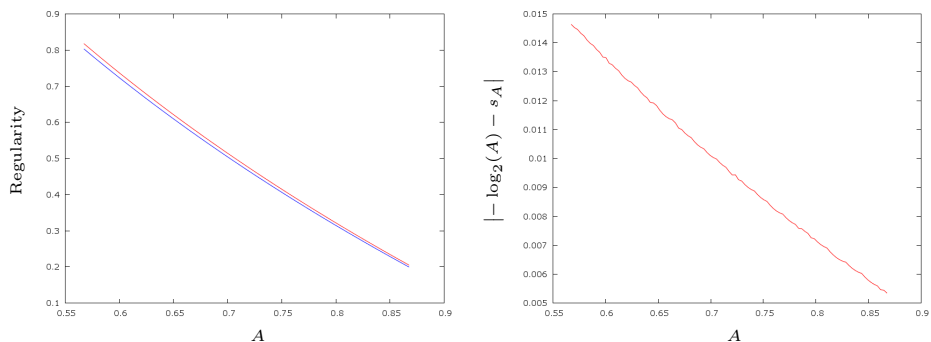
We want to remark that the best numerical estimate of the regularity of
$\mathfrak{W}_{A,2}(x)$ with $A \in [0.56\cdots,0.86\cdots]$ computed
with a Daubechies wavelet of 10 vanishing moments is obtained for
$A=0.86\cdots$ (that is, when the regularity is closer to zero). The
fact that we have to work with the Daubechies wavelet of 10 vanishing
moments can be explained as follows. Daubechies wavelets with higher
vanishing moments have bigger domain and regularity (see \cite{Mallat})
and, hence, they are less adapted to approximate the Weierstra\ss{}
function, which has highly concentrated oscillations. It turns out that
the value of 10 vanishing moments is the best adapted (in the sense that
minimizes the error) to the Weierstra\ss{} function for the range of
parameters considered.

We also want to remark that all the computed Pearson correlation
coefficients are bigger than 0.999. This agrees with the fact that the
Weierstra\ss{} function is self-similar. Then, the coefficients
$d_{j}[n]$, (as pointed out in Remark~\ref{LinearRegression}) must be
approximately on a straight line. This is what
Figure~\ref{regressioWeie} shows for a particular case.

\begin {figure}[htbp]
\begin{center}
\includegraphics[width=0.7\textwidth]{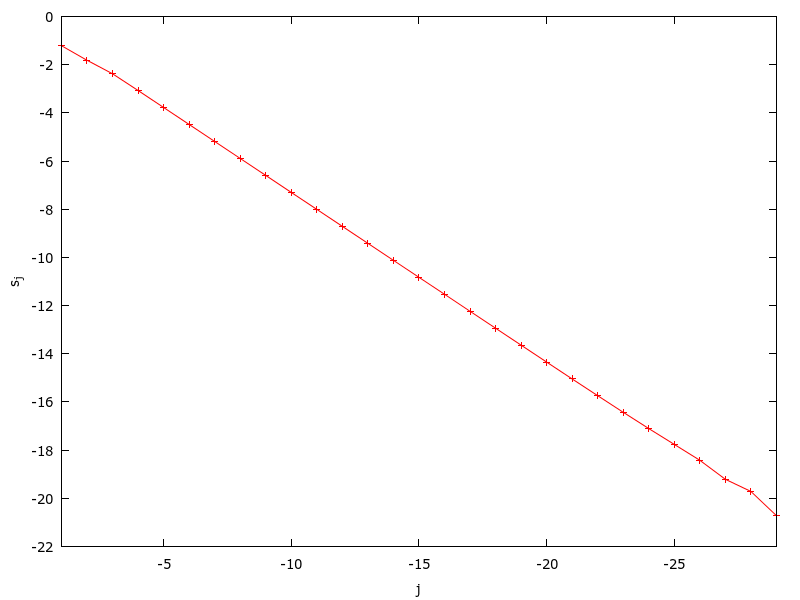}
\caption{The graph of the pairs $(j,s_j)$ with $-29 \le j \le 0$ for
$\mathfrak{W}_{0.86745,2}$ (the regularity
is $0.205147$).}\label{regressioWeie}
\end{center}
\end {figure}

It turns out that the Daubechies wavelet with $10$ vanishing moments
also maximizes globally the computed Pearson correlation coefficients in
the range of parameters that we consider. Hence, despite of the
inherited error of the FWT's seed, Daubechies wavelet with 10 vanishing
moments is the best to approximate and to explain the regularity of the
Weierstra\ss{} function. The error of such estimates are represented
on the left hand side of Figure~\ref{taulaweie}.
\end{example}
In view of the good results obtained in Example~\ref{exemnumWeier}, we
will perform the same methodology to a family of a Strange
Non-Chaotic Attractors in the following sections.
\section{An upper semi continuous \textsf{SNA}}\label{sec::Statement}
In \cite{GOPY}, a quasi-periodically forced skew product on the cylinder
was studied. The attractor obtained there (as shown by Keller
in \cite{Kell}), is the graph of an upper semi-continuous function.
Precisely, these kind of systems will be our testing grounds for the
algorithms that we are going to develop.

We start by introducing the model (and the attractor) that we are going
to study following~\cite{Kell}. We consider skew products on the
Cartesian product of the circle $\SI=\R\setminus\Z$ and
$\R^{+}=[0,\infty)$ of the type
\begin{equation}\label{equkelle}
\begin{pmatrix} \theta_{k+1} \\ x_{k+1}\end{pmatrix} =
\Fse (\theta_{k},x_{k}) =
\begin{pmatrix}
R_{\omega}(\theta_k) \\ f_\sigma(x_k)g_{\varepsilon}(\theta_k)
\end{pmatrix},
\end{equation}
where $(\theta_k,x_k) \in \SI\times\R^+$,
$R_{\omega}(\theta_k)=\theta_k+\omega\ \pmod{1}$ and $\omega \in
\R\setminus\mathbb{Q}$. On the second component, the map
{\map{g_{\varepsilon}}{\SI}[{[0,\infty)}]} is continuous (hence
bounded---for example $(\varepsilon+\abs{\cos(2\pi\theta)})$) and the
map {\map{f_\sigma}{[0,\infty)}} is $\Ccla{1}$, bounded, increasing,
strictly concave and such that $f_\sigma(0)=0$ (e.g.
$2\sigma\tanh(x)\evalat{\R^{+}}$). Observe that, since $f_\sigma(0)=0$,
the circle $x\equiv0$ is invariant.

Recall that the vertical Lyapunov Exponent at a point $(\theta_0,x_0)$
is defined by
\[
\limsup_{k\to\infty} \frac{1}{k}
  \log \norm{
     \begin{pmatrix}
        1 & 0 \\
        \dfrac{\partial x_k}{\partial \theta} & \dfrac{\partial
x_k}{\partial x}
     \end{pmatrix} \begin{pmatrix} 0 \\ 1 \end{pmatrix} } =
\limsup_{k\to\infty} \frac{1}{k} \log\abs{\dfrac{\partial x_k}{\partial
x}}.
\]
Therefore, by using Birkhoff's Ergodic Theorem, it can be shown that the
vertical Lyapunov Exponent at $x\equiv0$ is
\[
\kappa(f_\sigma,g_\varepsilon) :=
\int_{\SI} \log\left|
   \frac{\partial f_\sigma(x)g_{\varepsilon}(\theta)}{\partial
x}\biggr\rvert_{x=0}
\right| d\theta =
\log(f_{\sigma}'(0)) + \int_{\SI} \log \abs{g_{\varepsilon}(\theta)}
d\theta.
\]
When $\kappa(f_\sigma,g_\varepsilon)$ is positive, $x\equiv0$ is a
repellor for
System~\eqref{equkelle}. Moreover, since $f_{\sigma}$ and
$g_{\varepsilon}$ are bounded, infinity is also a repellor and the
system must have an attractor different from $x\equiv0$. These
attractors, which are the objects that we want to study, are typically
but not generally very complicate.

As we will see later, we are going to restrict ourselves to the study of
a particular subfamily of Model~\eqref{equkelle} which is
\begin{equation}\label{keller-GOPYe}
\begin{pmatrix} \theta_{k+1}\\x_{k+1}\end{pmatrix} =
\Fse (\theta_{k},x_{k}) =
\begin{pmatrix}
   R_{\omega}(\theta_k)\\
   2\sigma\tanh(x_k)\cdot(\varepsilon + \abs{\cos(2\pi\theta_k)})
\end{pmatrix},
\end{equation}
with $\omega = \tfrac{1+\sqrt{5}}{2}$, $\sigma > 0$ and $\varepsilon
\ge 0$. Apart from the parameter $\varepsilon$, it is the natural
restriction to $\R^{+}$ of the system considered in \cite{GOPY}  (see
Figure~\ref{dibuix-KGe}, where a graph of the attractor of this system
with $\sigma=1.5$ and $\varepsilon=0$ is shown).
\begin {figure}[htbp]
\begin{center}
\includegraphics[width=0.8\textwidth]{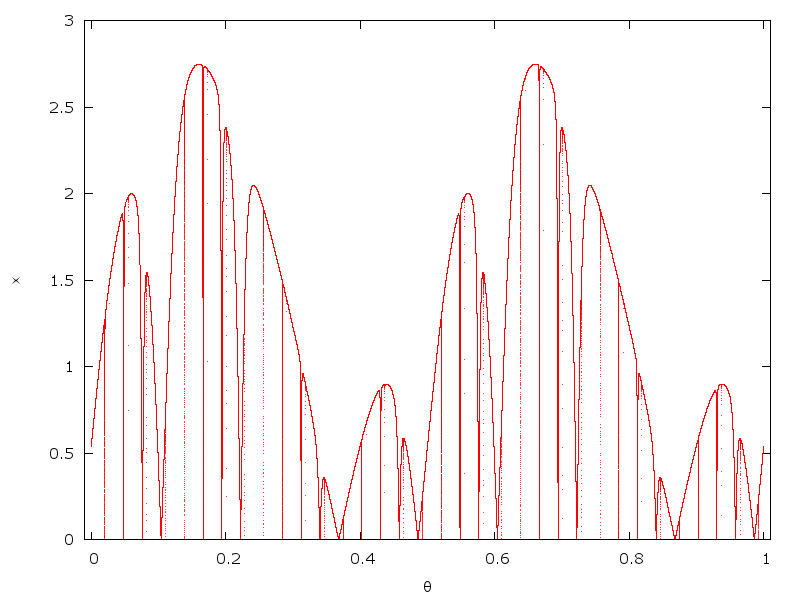}
\caption{The attractor of System \eqref{keller-GOPYe} for $\sigma=1.5$
and $\varepsilon=0$. Notice the abrupt changes in the graph of the
attractor.}\label{dibuix-KGe}
\end{center}
\end {figure}
In this case, the vertical  Exponent $\kappa(f_\sigma,g_\varepsilon)$ at
$x\equiv0$ is precisely $\log(\sigma)$. Hence, the interesting case (for
us) occurs when $\sigma > 1$.

The attractor of System~\eqref{equkelle} and its dynamics is described
by the following theorem (see also Figure~\ref{dibuix-KGe} for an
illustration of the graph of this attractor):

\begin{theorem}[G. Keller, \cite{Kell}]\label{thm:Keller}
There exists an upper semi continuous function
{\map{\varphi}{\SI}[\R^+]} whose graph is invariant under
System~\eqref{equkelle} and satisfies
\begin{enumerate}
\item The Lebesgue measure on the circle, lifted to the graph of
$\varphi$ is a Sinai-Ruelle-Bowen measure (that is,
\[
\lim_{k\to\infty} \dfrac{1}{k}
\sum_{i=0}^{k-1} f_\sigma(\Fse^{i} (\theta,x)) =
\int_{\SI} f_\sigma(\theta,\varphi(\theta))\ d\theta
\]
for every $f\in\Ccla{0}(\SI\times\R^{+},\R)$ and Lebesgue almost every
$(\theta,x)\in\SI\times\R^+$),

\item if $\kappa(f_\sigma,g_\varepsilon) \leq 0$ then $\varphi\equiv0$,

\item if $\kappa(f_\sigma,g_\varepsilon) > 0$ then $\varphi(\theta)>0$
for almost every
$\theta$,

\item if  $\kappa(f_\sigma,g_\varepsilon) > 0$ and $g$ vanishes at some
point then the
set
$\set{\theta \in \SI}{\varphi(\theta) > 0}$
is meager and $\varphi$ is almost everywhere discontinuous,

\item if $\kappa(f_\sigma,g_\varepsilon) > 0$ and $g>0$ then $\varphi$
is positive and
continuous; if $g\in\Ccla{1}$ then so is $\varphi$,

\item if $\kappa(f_\sigma,g_\varepsilon) \neq 0$ then
$\abs{x_{n}-\varphi(\theta_n)}\to 0$
exponentially fast
for almost every $\theta$ and every $x>0$.
\end{enumerate}
\end{theorem}

Observe that, when $\kappa(f_\sigma,g_\varepsilon) > 0$ and $g$ vanishes
at some point
(i.e. there exists a fibre whose image is degenerate to a point), it
follows from statements~(c,d) that $\varphi$ is discontinuous almost
everywhere. If there exists $\theta_0\in\SI$ such that
$g_{\varepsilon}(\theta_0)=0$,
we will say that the system is \emph{pinched}. In the particular case
of
System~\eqref{keller-GOPYe}, the pinching condition implies that
$\varepsilon = 0$ and, since $|\cos(2\pi\theta)|$ vanishes for $\theta
\in \left\{\tfrac{1}{4},\tfrac{3}{4}\right\}$, it follows that the set
\begin{equation}\label{denset}
  \left\{(\tfrac{i}{4}+k\omega \pmod{1},0)\ \colon\
       k\in\N,\  i\in\{1,3\}\right\}
\end{equation}
is both a subset of the attractor and is dense (and invariant) in
$x\equiv0.$ On the other hand, if $\varepsilon>0$ we can not have a
dense set of \emph{pinched} points.

The proof of the above theorem is based on the iteration of the
\emph{Transfer Operator} of the system and many of the properties of
$\varphi$ can be derived from such iteration. Since we will strongly use
this construction let us briefly explain it. Let $\mathscr{P}$ be the
space of all functions (not necessarily continuous) from $\SI$ to $\R$.
If we look for a functional version of the System~\eqref{equkelle} in
the space $\mathscr{P}$ one can define the \emph{Transfer Operator}
{\map{\mathfrak{T}}{\mathscr{P}}} as
\[
\mathfrak{T}(\varphi)(\theta) =
f_\sigma(\varphi(\Rot{\omega}{-1}(\theta))) \cdot
g_{\varepsilon}(\Rot{\omega}{-1}(\theta)).
\]
\begin{remark}\label{remarcaposada} From the above definition we obtain
\[
\mathfrak{T}(\varphi)(\theta) =
\pi_x\left(
   \Fse( \Rot{\omega}{-1}(\theta),
\varphi(\Rot{\omega}{-1}(\theta)) )
\right)
\]
where {\map{\pi_x}{\SI\times\R^+}[\R^+]} denotes the projection with
respect to the second component.
\end{remark}

Notice that the graph of a function {\map{\varphi}{\SI}[\R]} is
invariant for the System~\eqref{equkelle} if and only if
$\mathfrak{T}(\varphi) = \varphi$.

To obtain the map $\varphi$ from Theorem~\ref{thm:Keller}, Keller
takes a sufficiently large constant function $\varphi_0=c$ (with
$
c > \sup\limits_{x\in\R} f_\sigma(x)
    \max\limits_{\theta\in[0,1]} g_{\varepsilon}(\theta)
$) and iterates it under the transfer operator $\mathfrak{T}$ (see
Figure~\ref{transfer}). In such a way he gets, since the map $f$ is
monotone, a non-increasing  sequence of continuous maps given by
\begin{equation}\label{iterTrans}
\varphi_{k}=\mathfrak{T}(\varphi_{k-1})=\mathfrak{T}^{k}(c),
\end{equation}
where $\mathfrak{T}^{k}$ stands for the $k$-th iterate of the Transfer
Operator. Then, following Keller's proof, one has that
\[
\varphi:=\lim_{k\to\infty}\varphi_k=\inf_{k\to\infty}\varphi_k
\]
and the point-wise convergence of the above sequence is exponentially
fast. Notice that the shape of $\varphi$ depends on the parameters
$\sigma$ and $\varepsilon$.

\begin{figure}
\begin{center}
\includegraphics[width=0.6\textwidth]{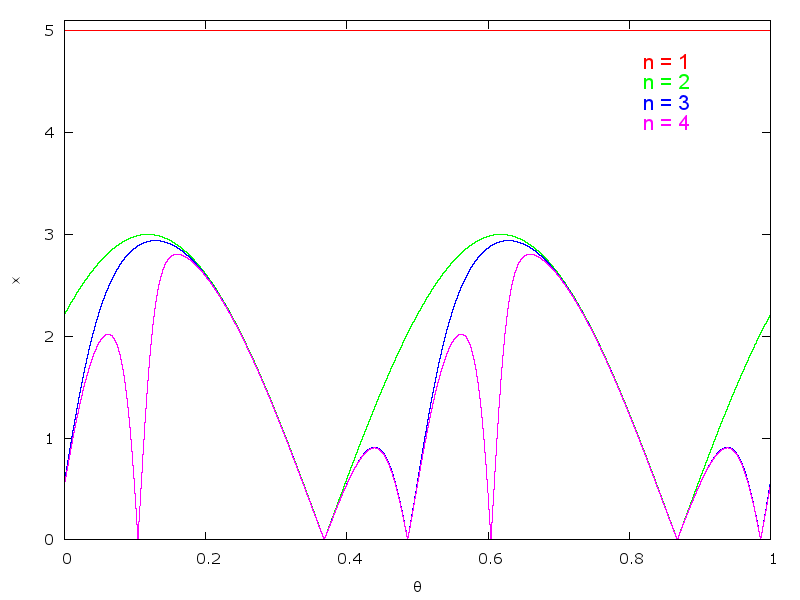}
\caption{The constant function $c=5$ and three iterations of the
Transfer Operator $\mathfrak{T}$ for System~\ref{keller-GOPYe} with
$\sigma=1.5$ and $\varepsilon=0$. The function $c$ is plotted in
\textcolor{red}{red}, $\mathfrak{T}(c)$ in \textcolor{green}{green},
$\mathfrak{T}^2(c)$ in \textcolor{blue}{blue} and $\mathfrak{T}^3(c)$
in \textcolor{magenta}{magenta}.}\label{transfer}
\end{center}
\end{figure}

Moreover, when proving the upper semi continuity of $\varphi$, it is
shown that all sets $\{\theta\in\SI\ \colon\
\varphi(\theta)<\varepsilon\}$ are open. This means that $\varphi$ is
continuous at each point where $g$ vanishes (whether pinched or
non-pinched) and, also, it is in $\Li(\SI)$. That is, the function
$\varphi\in\Li(\SI)$ is continuous at
\[
\zeron{n}:=\{ Z_{g_{\varepsilon}} + kw \pmod{1}\ :\ k=0,\ \dots,\
n\}=\bigcup_{k=0}^{n}R_{\omega}^{k}(Z_{g_{\varepsilon}})\subset\SI,
\]
where $Z_{g_{\varepsilon}}\subset\SI$ is the finite and discrete set
where the function $g_{\varepsilon}$ vanishes.

\subsection{On the regularity of the attractor}

The regularity of the attractor of System~\eqref{equkelle} in terms of
$\varepsilon$ and the regularity of $g_{\varepsilon}$ is given by the 
following

\begin{proposition}\label{Fita_s}
Let {\map{\varphi}{\SI}[\R^+]} be the upper semi continuous function 
whose graph is invariant under
System~\eqref{equkelle}.
Assume that $g_\varepsilon\in\Bes[s](\SI)$ with $s\in(0,1]$.
\begin{enumerate}
\item If $\varepsilon > 0$ then $\varphi\in\Bes[s](\SI).$
\item If $\varepsilon=0$ then $\varphi\in\Bes[0](\SI)$.
\end{enumerate}
\end{proposition}

\begin{proof}
Statement~(a) follows from~\cite[Theorem 1.2]{Star1}
and Statement~(b) follows from Theorem~\ref{thm:Keller} and 
Lemma~\ref{salt}.
\end{proof}

\begin{remark}
For $\varepsilon>0$ we can also have $\varphi\in\Bes[0](\SI)$.
Indeed, let $g_{\varepsilon}$ (in System~\eqref{equkelle}) be such that
\[
\lim_{\theta\to 0}\frac{|\theta|^\alpha}{g_{\varepsilon}(\theta)}=0.
\]
This means that $g_{\varepsilon}(\theta)$ does not verify any H\"{o}lder
condition but it is continuous (see
Example~\ref{alphanegex}\eqref{ex2}).
By Keller's Theorem, the invariant
function $\varphi$ is continuous when $\varepsilon>0$. However, by
the choice of $g_{\varepsilon}$, it cannot verify any H\"{o}lder
condition. Therefore, $\varphi\in\Bes[0](\SI)$ (even in the non-pinched
case).
\end{remark}
\section{An algorithm to compute the wavelet coefficients and
regularities of the attractors of
System~\eqref{equkelle}}\label{sec::extended_alg}
As it has been already said we want to approximate the invariant curve
$\varphi$ of the System~\eqref{keller-GOPYe} in terms of wavelets
with control quality. Since
$g_{\varepsilon}(\theta)=\varepsilon+\abs{\cos(2\pi\theta)}$
is Lipschitz then, by Proposition~\ref{Fita_s}, we know that
the regularity of the attractor is 0 when $\varepsilon = 0$ and 1
when $\varepsilon > 0$.
The control quality is implemented by comparing this theoretical value
with the estimate of the regularity obtained from the wavelet 
coefficients.

Since the attractor of System~\eqref{keller-GOPYe} is the
graph of a map {\map{\varphi}{\SI}[\R^+]} we will use the
methodology described in the previous section applied to the function
$\PER{\varphi}$ (see Lemma~\ref{PER} and Remark~\ref{pascual}). However,
$\PER{\varphi}$, in the pinched case,  is discontinuous almost
everywhere (and the corresponding attractor is called \emph{strange}).
Therefore, we are \textbf{not} allowed to apply verbatim the algorithm 
from
the previous section. In the rest of this section we will describe how
to solve this problem in the implementation of the strategy from the
previous section.

More concretely, to compute an approximation of the type
\eqref{final_approximation} for $\PER{\varphi},$ since we do not have an
explicit formula for $\varphi,$ we will use Theorem~\ref{thm:Keller}(f)
and the transfer operator to get a sufficiently good numerical
approximation of this function. Indeed, by Theorem~\ref{thm:Keller}(f),
for almost every $\theta_0\in\SI$, any $x_0>0$ and any $\varepsilon>0$
there exists $N_0$ such that for every $n\geq N_0$ we have:
\[
\abs{x_n - \varphi(\theta_{n})} < \varepsilon
\]
where $(\theta_n,x_n)=\Fse^{n}(\theta_0,x_0)$. Moreover, the
points $(\theta_n,x_n)$ with $n\in \{N_0,N_0+1,\dots,N_0+2^{J}-1\}$
approximate exponentially fast the points
$(\theta_n,\varphi(\theta_n))$ from $\mathrm{graph}(\varphi).$
Therefore,
\begin{equation}\label{data}
\set{(\theta_n,x_n)}{n = N_0,N_0+1,\dots,N_0+2^{J}-1}
\end{equation}
is an approximate mesh of $\mathrm{graph}(\varphi)$ provided that $J$
is large enough. To fix the mesh we choose a random point $\theta_0$
and we fix some $x_0 > \sup_{x\in\R^+} 2\sigma\tanh(x) = 2\sigma.$
However, this approximate mesh has two problems to be used in our
computations:
\begin{enumerate} [\textbf{Problem (1)}]
\item as we will see, we need a mesh of the graph of
$\varphi$ at dyadic points of the form
$i 2^{-J}$ for $i = 0,1,\dots, 2^{J}-1$,

\item we cannot use Lemma~\ref{FWT-InApprox} to estimate the initial
coefficients $a_{-J}[n]$ since our map $\varphi$ is discontinuous
almost everywhere (and, hence, not Lipschitz).
\end{enumerate}
In the following two subsections, we will explain how one can solve the
above two problems.
\subsection{A solution to Problem (1): a $\Ccla{1}$
homeomorphism}
As we have said, we need a mesh of the graph of $\varphi$ at dyadic
points of the form $\theta_{i} = i2^{-J}$ for $i = 0,1,\dots,2^{J}-1$
but, clearly, if we obtain the points $(\theta_n,x_n)$ just as iterates
of a single point by $\Fse$ this condition is not satisfied. The natural
approach which would be to approximately compute the points of the graph
of $\varphi$ based at the dyadic points by interpolating the obtained
values is not feasible since, by Theorem~\ref{thm:Keller}, we know that
$\varphi$ is upper semi-continuous and discontinuous everywhere. Then we
propose the following solution which consists in moving to a conjugate
system with the desired properties. To do this, first we relabel the
points $\{(\theta_n,x_n)\}_{n=N_0}^{N_0+2^{J}-1}$ to a sequence
$\{(\widetilde{\theta}_i,z_i)\}_{i=0}^{2^{J}-1}$ so that
\[
0 \leq \widetilde{\theta}_0 < \widetilde{\theta}_1 <
       \dots < \widetilde{\theta}_{2^{J} -1} < 1
\]
(we do this simply by sorting the data \eqref{data} with respect to the
first coordinate; see Remark~\ref{futnut}). In particular if
$n \in \{N_0,N_0+1,\dots,N_0+2^{J}-1\}$ and
$i = i(n) \in \{0,1,\dots,2^{J}-1\}$ is such that
$\widetilde{\theta}_i = \theta_n$, then $z_i = x_n.$

Now, we consider a $\Ccla{1}$ homeomorphism {\map{h}{\SI}} such that
$h\left(i 2^{-J}\right) = \widetilde{\theta}_i$ for $i =
0,\dots,2^{J}-1$, $h(1) = \widetilde{\theta}_0 + 1.$ Such map $h$ can
be obtained by taking $h$ to be, for instance, a cubic spline in the
intervals $\left[i 2^{-J}, (i+1)2^{-J}\right]$ for
$i = 0,\dots,2^{J}-1$.

Clearly $\left\{\left(i 2^{-J}, z_i\right)\right\}_{i=0}^{2^{J}-1}$ is
now an approximate mesh of $\mathrm{graph}(\varrho)$ with $\varrho
=\varphi \circ h,$ based at the dyadic points. Thus, \emph{we will use
the list of pairs
$\left\{\left(i 2^{-J}, z_i\right)\right\}_{i=0}^{2^{J}-1}$ to estimate
the regularity of $\varrho$}.

\begin{remark}\label{conjugacy}
The map $\varrho$ has the following dynamical interpretation. Consider
the homeomorphism {\map{H}{\SI\times\R^+}} defined by
$H(\theta,x) = (h(\theta),x)$. Then, one can check that,
$\mathrm{graph}(\varrho)$ is the attractor of the dynamical system
\[
\left( H^{-1}\circ\Fse \circ H\right) (\theta,x) =
(h^{-1}(R_{\omega}(h(\theta))),f_\sigma(x)g_{\varepsilon}(h(\theta))),
\]
which is conjugate to System~\eqref{keller-GOPYe}.
\end{remark}

To obtain the regularity of $\varphi$ we need to relate the regularities
of $\varrho$ and $\varphi$ in terms of the Besov Spaces $\Bes(\SI)$. To
do this we use the fact that $h$ is a $\Ccla{1}$ homeomorphism and that
a homeomorphism is a bijective mapping of $\Bes(\R)$ onto itself (we
refer the reader to Section~4.3 from \cite{Tri03} for a more detailed
explanation). More precisely,

\begin{proposition}\label{conjjectura}
Let $f\in\Bes(\SI)$ with $s\in\R$ and let {\map{h}{\SI}} be a
$\Ccla{m}$ diffeomorphism with $m \ge s$. Then $f\circ h$
belongs to $\Bes(\SI)$.
\end{proposition}

Thus, by Proposition~\ref{conjjectura} and Proposition~\ref{Fita_s}, the
regularity of $\varphi$ and $\varrho$ coincide and we can estimate the
regularity of $\varrho$ by using the mesh
$\left\{\left(i 2^{-J}, z_i\right)\right\}_{i=0}^{2^{J}-1}$.

We remark that the exact formula for $h$ is irrelevant for our
algorithm. We only use the fact that such a map $h$ exists and the fact
that it can be taken $\Ccla{1}.$ To obtain the data mesh $\left\{\left(i
2^{-J}, z_i\right)\right\}_{i=0}^{2^{J}-1}$ that approximates $\varrho$
(the attractor of the conjugate system) we simply have to sort the
obtained mesh $\{(\theta_n,x_n)\}_{n=N_0}^{N_0+2^{J}-1}$ for $\varphi$
with respect to the first coordinate $\theta$ and replace
$\widetilde{\theta}_i$ by $i 2^{-J}.$ Of course, this does not add any
further computational error to the mesh other than the errors coming
from the iteration of the system and truncation errors derived from the
choice of $J$. Furthermore, the exponential contraction of the system to
the attractor (see Theorem~\ref{thm:Keller}(f)) still holds for the
conjugate system, thus assuring that there is no loss of precision when
replacing $\varphi$ by $\varrho$ (see Subsection~\ref{sec:P2}).

\begin{remark}\label{futnut}
The process of sorting the data of an array of $2^{30}$ points from
$\SI\times\R^+$ (stored as pairs of \texttt{double} variables in
\texttt{C}) turns to be the bottleneck of the whole algorithm (and the
most time consuming task of the whole program). Moreover, even the
process of computing and filling the array with the initial mesh of the
function $\varphi$ already spends a ``visible'' amount of CPU time.
Indeed, the iteration, storing  and sorting process (with a standard
sort algorithm like Heapsort) of this data spends about 2200 CPU
seconds, with a remarkable variability which depends on the initial
sorting of the data, in a computer with a Xeon processor at 3 GHz and 32
Gb of RAM memory. In order to reduce the time elapsed in the sorting
process we use the following trick based on the fact that the dynamical
system generating the $\theta_i$ data is the irrational rotation
$R_\omega$. In this case we know that the Lebesgue measure is the unique
ergodic measure of $R_\omega$ and, hence, its averaged spatial
distribution is uniform and it is controlled approximately by the
Birkhoff's Ergodic Theorem applied to the Lebesgue measure. Indeed, we
have
\[
\sharp\left(
  \left\{\theta, R_{\omega}(\theta), \dots,
R_{\omega}^{k-1}(\theta)\right\}
  \cap \left[\tfrac{i}{N},\tfrac{i+1}{N}\right)
\right) \approx \frac{k}{N}
\]
for $k$ large enough and for every $i \in \{0,1,\dots, N-1\}.$ The
interpretation of this equation is that the statement
\begin{equation}\label{BET-orbit}
\sharp\left(
  \left\{\theta_{N_0}, \theta_{N_0+1}, \dots,
\theta_{N_0+2^J-1}\right\}
  \cap \left[\tfrac{i}{2^J},\tfrac{i+1}{2^J}\right)
\right) = 1
\end{equation}
holds with high frequency for $J$ large enough (observe that in this
case we have
$
 \left\{\theta_{N_0}, \theta_{N_0+1}, \dots,
\theta_{N_0+2^J-1}\right\} =
 \left\{\theta_{N_0}, R_{\omega}(\theta_{N_0}), \dots,
R_{\omega}^{2^J-1}(\theta_{N_0})\right\}
$%
). Moreover, when \eqref{BET-orbit} holds, we have
$i = \left\lfloor 2^J \theta_l \right\rfloor,$
where $\theta_l$ is the unique element from the set
$
  \left\{\theta_{N_0}, \theta_{N_0+1}, \dots,
\theta_{N_0+2^J-1}\right\}
  \cap \left[\tfrac{i}{2^J},\tfrac{i+1}{2^J}\right)
$
and $\lfloor \cdot \rfloor$ denotes the integer part function. This
observation gives a good ``hash function'' and the following efficient
algorithm to store and sort the data
$\{(\theta_n,x_n)\}_{n=N_0}^{N_0+2^{J}-1}.$ First, for
$n=N_0,N_0+1,\dots N_0+2^J - 1$ we compute the point
$(\theta_n,x_n)=\Fse(\theta_{n-1},x_{n-1})$. Then, we store it in the
position $i = \left\lfloor 2^J \theta_n \right\rfloor$ of the array
data, if this slot is free. Otherwise, we store the point
$(\theta_n,x_n)$ in a free position $j = j(i)$ of the array data such
that $\abs{j-i}$ is minimal. According to the above observations this
will happen with low frequency and the array data will be almost sorted.
Moreover, the positions of the array data which are not sorted are close
to the place where they should be when the array is sorted. This is
exactly the situation where the direct insertion sorting algorithm can 
be used with
very good results. This means that we are using a method of order
$\mathcal{O}(2^J+d)$ where $d$ is the number of insertions (which are
very low due to the way we have stored all data) instead of a method of
order $\mathcal{O}(J2^J)$ as the Heapsort algorithm.

With this trick, the  iteration, storing  and sorting process lasts
about 300 CPU seconds, almost without variability, which clearly
improves the efficiency of the program.
\end{remark}

\subsection{A solution to Problem (2): calculating the
coefficients $\PER{a}_{-J}[n]$ of $\PER{\varrho}$}\label{sec:P2}
We introduce the following notation for the wavelet coefficients of
$\PER{\varrho}$:
\[
\PER{a}_j[n] := \SP{\PER{\varrho},\phi_{j,n}}
  \quad\text{and}\quad
\PER{d}_j[n]:=\SP{\PER{\varrho},\psi_{j,n}}
\]
for $j,n \in \Z.$ Observe that, when $\varrho$ is regular enough,
Lemma~\ref{FWT-InApprox} gives
$2^{-J/2}\varrho\left(\tfrac{n}{2^{J}}\right)$ as an estimate for the
coefficients $\PER{a}_{-J}[n].$ But, as we have pointed out, $\varphi$
(and hence $\PER{\varrho}$) is discontinuous almost everywhere and the
above estimate of $\PER{a}_{-J}[n]$ is, a priori, not valid. However, as
we will see, the element $z_n \approx \varrho\left(n 2^{-J}\right)$ from
our data give indeed a good estimate for $\PER{a}_{-J}[n]$ because our
mesh is based at the dyadic points $n 2^{-J}.$

As it has been already said in Section~\ref{sec::Statement}, $\varphi$
is the point-wise limit of a non-increasing sequence of continuous
(and, hence, uniformly continuous) functions
{\map{\varphi_{k}}{\SI}[\R^+]} defined by
\[
 \varphi_{0}(\theta) = c
 \qquad\text{and}\qquad
 \varphi_{k+1}(\theta) = \mathfrak{T}(\varphi_k)(\theta)
\]
for every $\theta \in \SI$ and $c > \sup_{x\in\R^+} 2\sigma\tanh(x) = 
2\sigma.$
Consequently,
$\varrho(\theta) = \lim_{k\to\infty}  \varphi_{k}(h(\theta))$
for every $\theta$.

\begin{remark}\label{iter_transfer}
\emph{If we take $x_0 = c = \varphi_0(\theta_0)$ then $x_k =
\varphi_k(\theta_k)$ for every $k \ge 1$}. To see this notice that, from
the definition of the points $(\theta_n,x_n)$ and $\Fse$, we get
\[
\theta_k = R_{\omega}(\theta_{k-1})
\qquad\text{and}\qquad
x_k = \pi_x(\Fse(\theta_{k-1},x_{k-1}))
\]
for every $k \ge 1$. Now, we proceed by induction. We
assume that $x_{k-1} = \varphi_{k-1}(\theta_{k-1})$ fore some $k \ge 0$.
Then, by Remark~\ref{remarcaposada},
\begin{align*}
x_k &= \pi_x(\Fse(\theta_{k-1},x_{k-1}))
     = \pi_x(\Fse(\theta_{k-1},\varphi_{k-1}(\theta_{k-1})))
\\
    &= \mathfrak{T}(\varphi_{k-1})(R_{\omega}(\theta_{k-1}))
     = \varphi_k(\theta_k).
\end{align*}
\end{remark}

Since the scaling function $\phi$ of a Daubechies wavelet is continuous,
so is $\phi_{-J,n}$ for each $n$. Hence, from the definition of the
coefficients $\PER{a}_{-J}[n]$ and the Dominated Convergence Theorem we
have:
\newcommand{\akper}{a^{k,{\scriptscriptstyle\mathrm{PER}}}_{-J}}
\begin{align*}
\PER{a}_{-J}[n]
  &= \int_{\mathrm{supp}(\phi_{-J,n})}
      \PER{(\varphi\circ h)}(\theta) \phi_{-J,n}(\theta)\ d\theta \\
  &= \lim_{k\to\infty} \int_{\mathrm{supp}(\phi_{-J,n})}
      \PER{(\varphi_k\circ h)}(\theta) \phi_{-J,n}(\theta)\ d\theta \\
  &= \lim_{k\to\infty} \akper[n],
\end{align*}
where
$
 \akper[n] :=  \SP{\PER{(\varphi_k\circ h)}, \phi_{-J,n}}.
$
From the proof of the Dominated Convergence Theorem, it can be shown
that $\akper[n]$ converge exponentially fast to $\PER{a}_{-J}[n]$.
Therefore, if $k$ is large enough, by Lemma~\ref{FWT-InApprox} we have
\[
\PER{a}_{-J}[n]\sim
\akper[n] \approx 2^{-J/2}\PER{(\varphi_k\circ h)}(n 2^{-J}) =
2^{-J/2}\varphi_{k}(h(n 2^{-J}))
\]
for $n = 0,\dots,2^{J}-1$ (where $\sim$ means exponentially close).

From the definition of $h$ it follows that, given $n \in
\{0,1,\dots,2^{J}-1\},$ there exists $k \in
\{N_0,N_0+1,\dots,N_0+2^{J}-1\}$ such that $h\left(n 2^{-J}\right) =
\widetilde{\theta}_n = \theta_k$. Therefore, by
Remark~\ref{iter_transfer},
\[
\varphi_k\left(h\left(n 2^{-J}\right)\right) =
  \varphi_k(\theta_k) = x_k = z_n.
\]
Hence, if $N_0$ is large enough,
\begin{equation}\label{finalz}
\PER{a}_{-J}[n]
 \approx 2^{-J/2}\varphi_{k}(h(n 2^{-J})) = 2^{-J/2} z_n
\end{equation}
for $n = 0,\dots,2^{J}-1$. This gives the necessary approximation of
the coefficients $\PER{a}_{-J}[n]$ to initialize the algorithm.

\subsection{The algorithm for the
System~\eqref{equkelle}}\label{sec::arasi_alg}
In view of the previous sections and Section~\ref{sec::TheAlgorithm}, we
present the algorithm to estimate regularities, in terms of the Besov
spaces $\Bes[s]$, for the System~\eqref{equkelle}.

\begin{algorithm}\label{algfinal}
Let $\Fse$ be a skew product under the assumptions
of the System~\eqref{equkelle}. Fix $J>0$ and a transient $N_0>0$ big
enough. To estimate the regularity of the invariant function,
$\varphi$, of the System~\eqref{equkelle} perform the following steps:
\step{1} \emph{Generation of a mesh of points exponentially close to
$\varphi$}. By Theorem~\ref{thm:Keller}(f), fix $\sigma > 1$, choose a
random $\theta_0 \in [0,1)$ and $x_0 > 1$ and, by using the recurrence
$(\theta_n,x_n)=\Fse(\theta_{n-1},x_{n-1}),$ generate the data
\[
\set{(\theta_n,x_n)}{n = N_0,N_0+1,\dots,N_0+2^{J}-1}.
\]

\step{2} \emph{Sort the data}. Using Remark~\ref{futnut}, sort the
above data to obtain a sequence
$\{(\widetilde{\theta}_n,z_n)\}_{n=0}^{2^{J}-1}$ so that
\[
0 \leq \widetilde{\theta}_0 < \widetilde{\theta}_1 <
       \dots < \widetilde{\theta}_{2^{J} -1} < 1,
\]
and delete the concrete values of the points $\widetilde{\theta}_n.$
This defines a map $\varrho$, using Remark~\ref{conjugacy} and
Proposition~\ref{conjjectura}, with the same regularity that the map
$\varphi$ such that $\varrho\left(n 2^{-J}\right) \approx z_n$ for $n =
0,\dots,2^{J}-1$.

\step{3} \emph{FWT's initialization}. Set $\PER{a}_{-J}[n]:=2^{-J/2}
z_n$ for $n = 0,\dots,2^{J}-1$. By Equation~\eqref{finalz} this is a
good approximation of the coefficients $\PER{a}_{-J}[n]$.

Now, since we have an approximation of $\PER{a}_{-J}[n]$, Steps~2--5 of
the strategy performed on Section~\ref{sec::TheAlgorithm} remain
unaltered. That is,

\step{4} \emph{FWT procedure}. Use Equation~\eqref{FWT} to calculate
the coefficients
$
\PER{d}_{-j}[n] = \SP{\PER{f},\psi_{-j,n}}
$
for $j=0,\dots,J-1$ and $0 \leq n \leq 2^{j}-1$.

\step{5} \emph{Application of Theorem~\ref{util}: data to compute the 
regularity}.
By using the coefficients $\PER{d}_{-j}[n]$ from Step~4, calculate
\[
s_{-j} = \log_{2}\left(
   \sup_{0\leq n \leq2^{j}-1} \abs{\PER{d}_{-j}[n]}
\right)
\]
for $j=0,\dots,J-1$.
\step{6} \emph{Application of Remark~\ref{LinearRegression}: compute the 
regularity}.
Make a linear regression to estimate the slope $\tau$ of the
graph of the pairs $(-j,s_{-j})$ with $j=0,\dots,J-1$. Then, when there
is evidence of linear correlation between the variables $-j$ and
$s_{-j}$, we set $s = \Reg(\tau).$

\step{7} \emph{Final test of assumptions of Theorem~\ref{util} on the 
number of vanishing moments}.
If $k > \max(s,5/2-s)$ then $f\in\Bes(\R)$ and, hence, $f$ has
regularity $s$. Otherwise we need to repeat Step 4 -- 7 with a
Daubechies wavelet having a larger value of $k$ until $k >
\max(s,5/2-s).$
\end{algorithm}

\begin{remark}\label{remalg}
We want to underline that the above algorithm can be performed for other
systems. The reason is that the ``\emph{regularity steps}'',
\textbf{Steps 4--7}, are valid for $f\in\Liin(\SI)$ due to 
Theorem~\ref{util}. Also
notice that \textbf{Steps 1--3} can be skipped if the wavelet 
coefficients
are obtained using other methods such as solving the Invariance Equation
for example. Indeed, recall that the graph of a function
{\map{\varphi}{\SI}[\R]} is invariant for the System~\eqref{equkelle} if
and only if $\varphi$ is a fixed \emph{point} of $\mathfrak{T}$. That
is:
\[
f_{\sigma}(\varphi(\Rot{\omega}{-1}(\theta))\cdot
g_{\varepsilon}(\Rot{\omega}{-1}(\theta))=
\mathfrak{T}(\varphi)(\theta) = \varphi(\theta).
\]
The above equation is called the \emph{Invariance Equation}. To give
an approximation of $\varphi$ on a certain mesh of points
$\theta_i\in\SI$, one can solve a non-linear system of equations by
imposing
\[
\varphi(\theta_i) =
a_0+\sum_{j=0}^{N}\sum_{n=0}^{N_j}d_{j,n}\psi_{j,n}
(\theta_i)
\]
on the Invariance Equation.
The unknowns, like in the Fast Wavelet Transform, are the wavelet
coefficients.
\end{remark}

As a result, we can get an estimate of the regularity of the (strange)
attractor of System~\eqref{keller-GOPYe} for the chosen value of
$\sigma$ and $\varepsilon.$ This will be the main topic of the following
section.
\section{Results and conclusions}\label{sec::Conclusions}
We have performed an exercise which is divided in two steps. Our
testing ground will be the System~\eqref{keller-GOPYe}. The range of
values for $\sigma$ is $[1,2]$ whereas $\varepsilon$ is given by the
function
\[
\varepsilon(\sigma)=\begin{cases}
    (\sigma-1.5)^2 & \text{when $1.5< \sigma \le 2,$}\\
           0       & \text{when $1 \le \sigma \le 1.5.$}
                    \end{cases}
\]
That is, we will use the System
\[
\begin{pmatrix} \theta_{k+1}\\x_{k+1}\end{pmatrix} =
\mathfrak{F}_{\sigma,\varepsilon(\sigma)} (\theta_{k},x_{k}) =
\begin{pmatrix}
   R_{\omega}(\theta_k)\\
   2\sigma\tanh(x_k)\cdot(\varepsilon(\sigma) +
    \abs{\cos(2\pi\theta_k)})
\end{pmatrix}.
\]
Notice that with this parametrization, the above system is pinched if
and only if $\sigma \in [1, 1.5]$. Roughly speaking, the parameters
$(\sigma,\varepsilon(\sigma))$ control the vertical Lyapunov Exponent of
$x\equiv0$ and the pinching condition at the same time. Hence, in view
of Proposition~\ref{Fita_s} we can test the quality of the wavelet
coefficients given by Algorithm~\ref{algfinal}. Indeed, we know that
for $\sigma\in[1.5,2.0]$ the estimated value of $s$ must be close to
$1$; whereas for $\sigma\in[1.0,1.5]$ the regularity parameter $s$ must
be zero. Moreover, when $\sigma$ crosses $1.5$, in a decreasing way,
the regularity parameter $s$ has to jump from $1$ to $0$. Hence, the
results obtained by Algorithm~\ref{algfinal} applied to our problem
and the fact that we can observe the jump of $s$ at $\sigma=1.5$
certify the quality of the algorithm that we have developed. Hence, we
can have a degree of accuracy of the wavelet coefficients.

To do this test, first we have applied Algorithm~\ref{algfinal} with
$N_0=10^{5}$ and $J = 30$. In Figure~\ref{directa} we plot the estimated
regularities of System~\eqref{keller-GOPYe} as a function of $\sigma$
with the above parametrization. Notice that, the estimated regularity
detects the jump at $\sigma=1.5$ in a correct way and agrees with the
values described above despite of the fact that the concrete values
of the regularity are not correct for $\sigma \gtrapprox 1.5.$
\begin{figure}[ht]
\begin{center}
\includegraphics[width=0.7\textwidth]{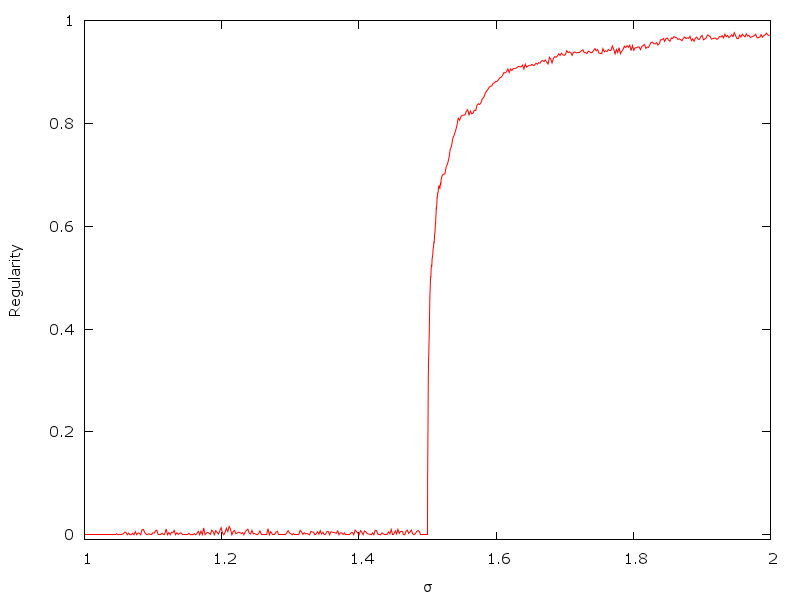}
\caption{The estimate of the regularity $\Reg(\widetilde{s})$ of the
(strange) attractor of System~\eqref{keller-GOPYe} for $\sigma \in
[1,2]$ and $\varepsilon$ given by the parametrization
$\varepsilon(\sigma)$. The results are obtained by using a sample of
$2^{30}$ points (that is, $J=30$), a transient $N_0=10^{5}$ and the
Daubechies wavelet with 16 vanishing moments. For this number of
vanishing moments we obtain the minimum variance of Pearson correlation
coefficient.}\label{directa}
\end{center}
\end{figure}

The choice of the Daubechies wavelet that we will use must be done 
carefully. Indeed,
recall that Daubechies wavelets must have more than $\max(s,5/2-s)$
vanishing moments in order to be under the assumptions of
Theorem~\ref{util}. However, the increase of the number of vanishing
moments, $k$, causes an increment of the support size of the wavelet
(see~\cite[Theorem 7.3]{Mallat}). On the other side, the support of the
map $\varphi$ is $\SI=\R\setminus\Z=[0,1)$. Observe that, the ratio of
growth between the support of $\varphi$ and $\psi_{-j,n}$ is $1$ to
$\tfrac{2k-1}{2^{j}}$. Therefore, there exists $j_0$ such that for
$j>j_0$ the support of $\psi_{-j,n}$ is contained in $[0,1)$. In view of
that, the first coefficients of the wavelet approximation will be
affected by an error induced by the size of the support of $\psi$. To
avoid such ``\emph{bad performance}'', a good strategy is to choose
different wavelets for different parameters as it can be guessed in
Figure~\ref{cutre01}. Having said that, Figure~\ref{directa} (and its
forthcoming discussion) it is done with Daubechies wavelet with 16
vanishing moments. This  choice is based in two reasons. The first one
is that all Daubechies wavelets used generate a similar picture as
Figure~\ref{directa}. On the other side, such wavelet explains
``\emph{better}'' all the range of $\varepsilon$ (even the close to zero
case as it can be seen in Figure~\ref{cutre01}).

\begin{figure}[ht]
\begin{center}
\includegraphics[width=0.7\textwidth]{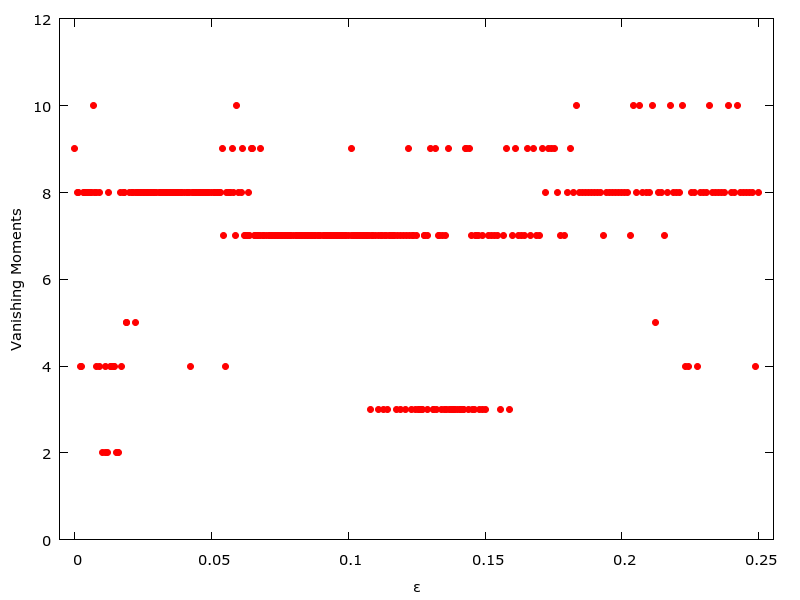}
\caption{The dot at level $n$ above an $\varepsilon$-value means that 
the
Pearson correlation coefficient applying Algorithm~\ref{algfinal} with
Daubechies wavelets of $2n$ vanishing moments is the biggest one (always
greater than $0.99$ and biggest in comparison with the Daubechies
wavelets of $k\neq 2n$ vanishing moments).}\label{cutre01}
\end{center}
\end{figure}

\begin{remark}
Figure~\ref{cutre01} has another interpretation. Indeed, the Daubechies
wavelets with 14 and 16 vanishing moments explain (in terms of
Equation~\eqref{LineqReg}) almost all the range of values of
$\varepsilon$. However, there are regions of $\varepsilon$ where other
wavelets are better than these ones. This means that a good strategy to
perform Algorithm~\ref{algfinal} in a more accurate way is to have a
\emph{dictionary} of wavelets. Notice that such \emph{dictionary}
cannot exists in the Fourier setting.
\end{remark}

The second step of the exercise, and in view of the results displayed
in Figure~\ref{directa}, is the explanation of the three regions that
appear. Indeed, in Figure~\ref{directa} one can clearly appreciate three
regions with different qualitative behaviour. One of them corresponds to
the pinched case (i.e. $\sigma \in [1,1.5]$) and the other two to the
non-pinched one: $\sigma \in (1.5,\widetilde{\sigma})$ and $\sigma \in
[\widetilde{\sigma},2]$ with $\widetilde{\sigma} \approx 1.527$. In what
follows we discuss in detail these three regions.

\subsubsection*{Non pinched case: $\sigma\in[\widetilde{\sigma},2]$}
In this region we have $\varepsilon = (\sigma-1.5)^2 \gtrapprox 7.29
\times 10^{-4}$. That is, we are ``\emph{far}'' from the pinched case.
As we already know, see Proposition~\ref{Fita_s}, the function $\varphi$
whose graph is the attractor is continuous but not differentiable.
Moreover, since we are far from the pinched case, $\varphi$ is rather
well behaved since we have lack of differentiability only in few points
(see the left picture of Figure~\ref{nopinch}). This is confirmed by
the estimated regularities that, not surprisingly (see
Proposition~\ref{Fita_s}), are in the interval $(0,1)$ and ``far'' from
zero: $\Reg(\widetilde{s})\in[0.6822,0.9669]$ (see the right part of
Figure~\ref{directa}).

\begin{figure}[ht]
\begin{center}
\hfill \subfigdef{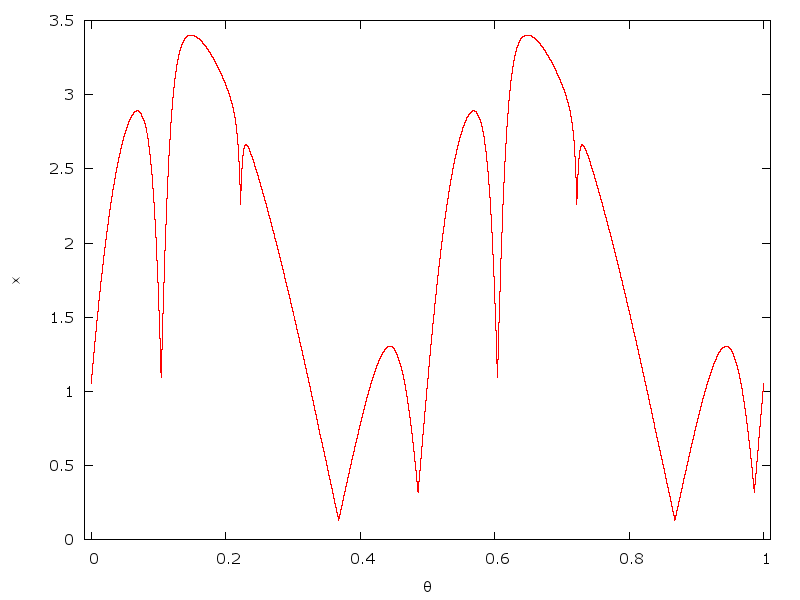}{$\varphi$ for $\sigma =
1.699219$ (i.e $\varepsilon = 0.039688$).}
\hfill \subfigdef{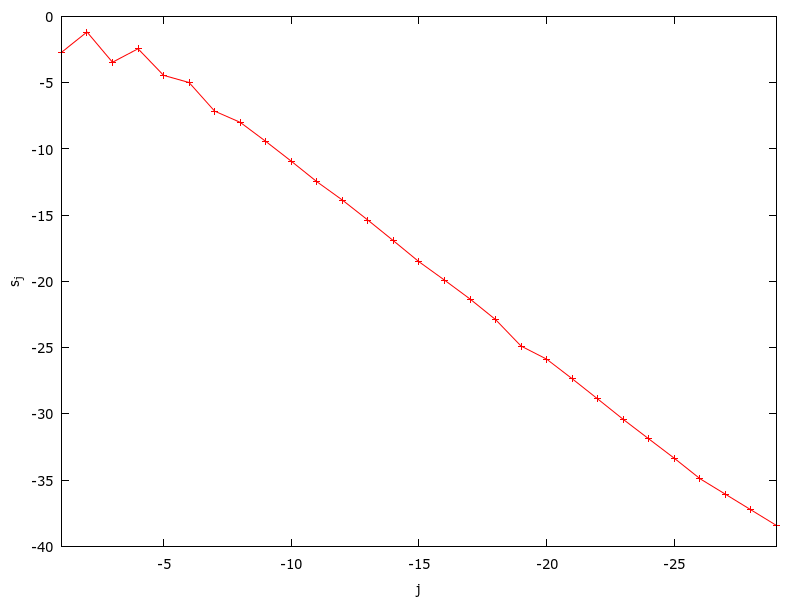}{$R(\tau)$ for $\sigma = 1.699219$
(i.e $\varepsilon = 0.039688$).}
\end{center}
\caption{In the left hand side it is plotted the attractor of
System~\eqref{keller-GOPYe} for $\sigma=1.699219$ (and
$\varepsilon=0.039688$). Whereas in the right hand side it is plotted
\emph{its} pairs $(j,s_j)$ with $-29 \le j \le 0$. In this case
$\Reg(\widetilde{s})=0.91431$.}\label{nopinch}
\end{figure}

Observe that (see the right picture of Figure~\ref{nopinch}) in
agreement with the computed Pearson correlation coefficient, the model
given by Equation~\eqref{LineqReg} is approximately linear (as we
expect). Also, some few of the first coefficients are not well fitted
(for the linear model) because of the aforesaid problem with the
support of $\psi$.

\subsubsection*{The pinched case: $\sigma\in[1,1.5]$}
In this case $\varepsilon =0$. Therefore, according to
Theorem~\ref{thm:Keller} and Proposition~\ref{Fita_s}, the attractor is
pinched and, hence, discontinuous almost everywhere. That is,  the
function $\varphi$ whose graph is the attractor is an upper semi
continuous function (see the left picture of Figure~\ref{sipinch}).
Thus, in agreement with Proposition~\ref{Fita_s}, the estimated
regularity is equal to zero for the whole range of parameters as it can
be guessed in the left part of Figure~\ref{directa}.

\begin{figure}[ht]
\begin{center}
\hfill \subfigdef{Keller-GOPY}{$\varphi$ for $\sigma =
1.425781$ (i.e $\varepsilon = 0.0$).}
\hfill \subfigdef{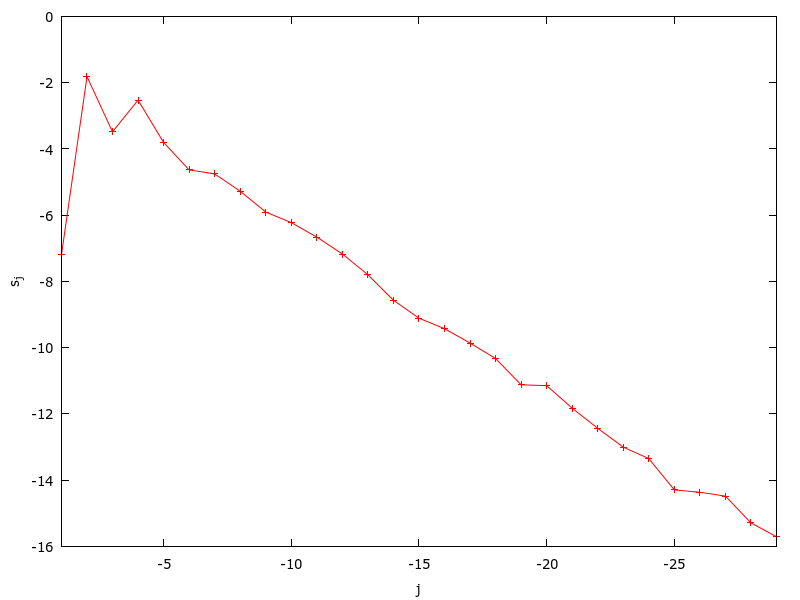}{$R(\tau)$ for $\sigma =
1.425781$ (i.e $\varepsilon = 0.0$).}
\end{center}
\caption{In the left hand side it is plotted the attractor of
System~\eqref{keller-GOPYe} for $\sigma=1.425781$ (and
$\varepsilon=0.0$). Whereas in the right hand side it is plotted
\emph{its} pairs $(j,s_j)$ with $-29 \le j \le 0$. In this case
$\Reg(\widetilde{s})=0.0149$.}\label{sipinch}
\end{figure}

Moreover, even being out of the hypothesis of Lemma~\ref{FWT-InApprox}
the coefficients obtained are almost linear. Indeed, from
Theorem~\ref{util}, the model given by Equation~\eqref{LineqReg} has
more freedom because there is a gap in $[-1/2,1/2]$ (see
Theorem~\ref{util}). However, see the right picture of
Figure~\ref{sipinch}, the model is linear (except for a few first
values as before). This means that the proposed solutions of Problem
(1) and (2) does not add error on our computations (as the case of the
left picture of Figure~\ref{sipinch}).

\subsubsection*{Approaching pinching case:
$\sigma \in (1.5,\widetilde{\sigma})$}
In this region we have $\varepsilon = (\sigma-1.5)^2 \lessapprox 7.29
\times 10^{-4}$. That is, we are ``\emph{close}'' to the pinched case.
Therefore, since $g_{\varepsilon}=\varepsilon+\abs{\cos(2\pi\theta)}$ is
Lipschitz, by Proposition~\ref{Fita_s}, the regularity must have a jump
from $0$ to $1$. Thus, in the estimated regularities one must
perceive such jump. Having said that, the estimated regularities in
this region are not so good (see the caption of Figure~\ref{nsipinch}).
However, looking at Figure~\ref{directa} and~\ref{nsipinch}, the pass
from $0$ to $1$ in a ``\emph{fast way}'' is still observed.

\begin{figure}[ht]
\begin{center}
\hfill \subfigdef{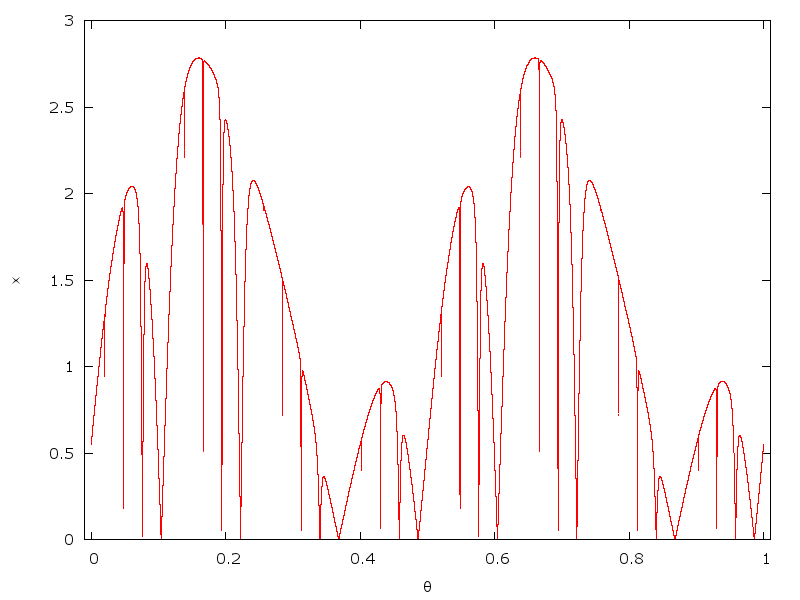}{$\sigma = 1.513672$
($\varepsilon = 0.000187$).}
\hfill \subfigdef{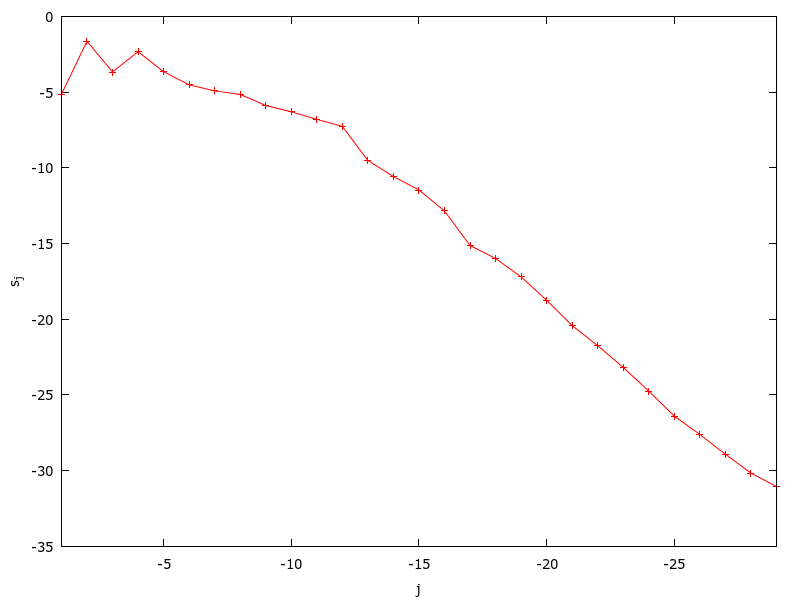}{$\sigma = 1.513672$
($\varepsilon = 0.000187$).}
\hfill \strut \\
\hfill \subfigdef{KGOPYQP1}{$\sigma = 1.507812$
($\varepsilon = 0.000061$).}
\hfill \subfigdef{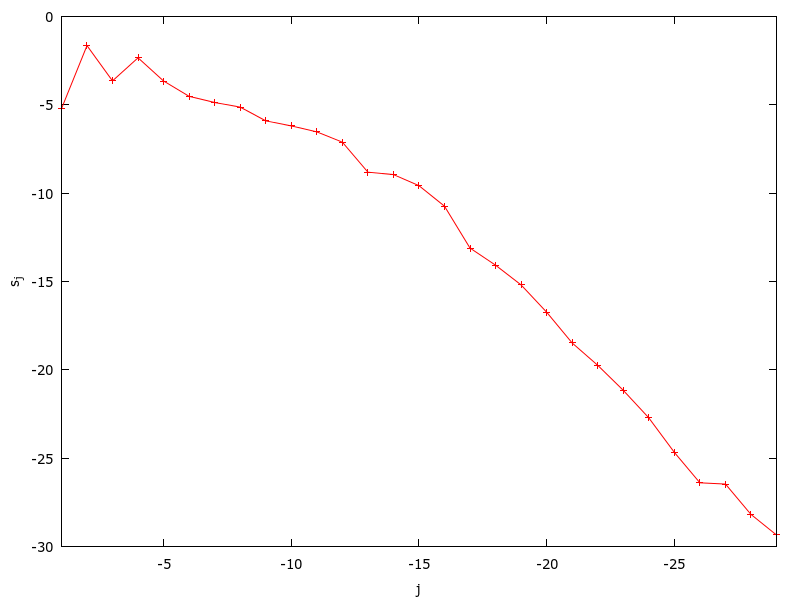}{$\sigma = 1.507812$
($\varepsilon = 0.000061$).}
\end{center}
\caption{In the left hand side it is plotted the attractor of
System~\eqref{keller-GOPYe} for two instances of $\sigma\in
(1.5,\widetilde{\sigma})$. Whereas in the right hand side it is plotted
\emph{its} pairs $(j,s_j)$ with $-29 \le j \le 0$. The estimated
regularity of the attractors is $\Reg(\widetilde{s})=0.6266$ and
$\Reg(\widetilde{s})=0.4951$ respectively.}\label{nsipinch}
\end{figure}

Observe that, as in the other cases, the first few values of the 
supremums $s_j$
are the worst fitted (see the right pictures
in Figure~\ref{nsipinch}). But, in contrast with the previous regions, 
the
rest of values of $s_j$ are not so ``\emph{well behaved}''. 
Nevertheless, they
have a big Pearson correlation coefficient. This ``\emph{bad 
behaviour}''
is probably due to a big value of the constant $C>0$ of 
Theorem~\ref{util}.
Indeed, for such range of values, we are close to the change of space
from positive regularity to zero. That is, from
Remark~\ref{LinearRegression} the constant $C>0$ tends to infinity when
$\varphi$ is approaching to the pinched case. This is precisely what it
is shown in Figure~\ref{cutre02}.

\begin{figure}[ht]
\begin{center}
\includegraphics[width=0.45\textwidth]{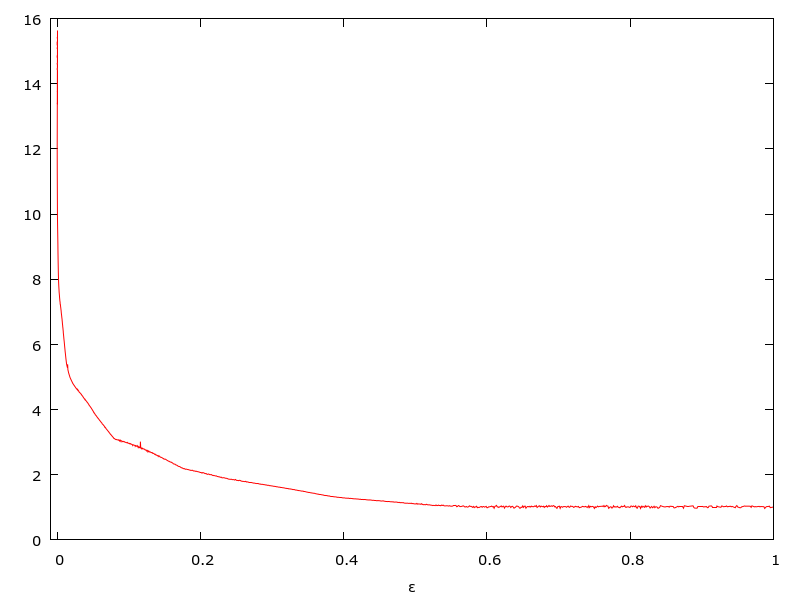}
\caption{The plot of the ``\emph{functional space jump}'', in terms of
the norm of  the attractor of System~\eqref{keller-GOPYe}, for
$\varepsilon\in(0,1]$ and $\sigma=1.5$. The constant $C>0$ becomes
unbounded as Remark~\ref{LinearRegression} asserts.}\label{cutre02}
\end{center}
\end{figure}

Finally, recall that the Fast Wavelet Transform is not the unique way
to obtain the wavelet coefficients.
In a forthcoming paper we will explore
the technique of solving numerically the Invariance Equation given in
Remark~\ref{remalg} which can be more adapted to the dynamical 
complexity
of the object and to better recover the large set of zeros of the 
\textsf{SNA}.
Moreover, it is also interesting to explore other models and understand 
other
``routes to complexity'' for the \textsf{SNA}'s.
In particular the study of the arc length curve (see \cite{JoTa})
or the Hausdorff dimension (see \cite{GrJa}) by means of wavelets
can help understanding some of these routes to \emph{strangeness}.
\bibliographystyle{alpha}
\bibliography{bibliografiaFWT}
\end{document}